\documentclass[twoside,11pt]{article}
\usepackage{jmlr2e}
\usepackage{blindtext}
\usepackage{amsmath,amssymb}
\usepackage{algorithm}
\usepackage{algorithmic}
\usepackage{subfigure}
\usepackage{color}
\usepackage{epstopdf}
\usepackage{threeparttable}

\newtheorem{defi}{Definition}
\newtheorem{ass}{Assumption}
\newtheorem{lem}{Lemma}
\newtheorem{thm}{Theorem}

\newtheorem{rem}{Remark}

\usepackage{lastpage}
\jmlrheading{?}{2025}{1-\pageref{LastPage}}{?; Revised ?}{?}{?}{Jun-Lin Wang, Zi Xu, and Hui-Ling Zhang}

\ShortHeadings{FF-CR Algorithm for Convex-Concave Minimax Problems}{Wang, Xu and Zhang}
\firstpageno{1}

\begin{document}

\title{A Fully Parameter-Free Second-Order Algorithm for Convex-Concave Minimax Problems\footnotemark[1]}

\author{\name Jun-Lin Wang \email wjl37@shu.edu.cn \\
	\addr Department of Mathematics\\
	Shanghai University\\
	Shanghai 200444, People's Republic of China
	\AND
	\name Zi Xu\footnotemark[2] \email xuzi@shu.edu.cn \\
	\addr Department of Mathematics\\
	Shanghai University\\
	Shanghai 200444, People's Republic of China\\
	and\\
	Newtouch Center for Mathematics of Shanghai University\\
	Shanghai University\\
	Shanghai 200444, People's Republic of China
	\AND
	\name Hui-Ling Zhang \email zhanghl@amss.ac.cn \\
	\addr LSEC, ICMSEC, Academy of Mathematics and Systems Science\\
	Chinese Academy of Sciences \\
	Beijing 100190,  People's Republic of China
	}

\renewcommand{\thefootnote}{\fnsymbol{footnote}}
\footnotetext[1]{This work is supported by National Natural Science Foundation of China under the grants 12471294. }
\footnotetext[2]{Corresponding author.}

\editor{My editor}

\maketitle

\begin{abstract}In this paper, we study second-order algorithms for the convex-concave minimax problem, which has attracted much attention in many fields such as machine learning in recent years. We propose a Lipschitz-free cubic regularization (LF-CR) algorithm for solving the convex-concave minimax optimization problem without knowing the Lipschitz constant. It can be shown that the iteration complexity of the LF-CR algorithm to obtain an $\epsilon$-optimal solution with respect to the restricted primal-dual gap  is upper bounded by $\mathcal{O}(\rho^{2/3}\|z_0-z^*\|^2\epsilon^{-2/3})$ , where $z_0=(x_0,y_0)$ is a pair of initial points, $z^*=(x^*,y^*)$ is a pair of optimal solutions, and $\rho$ is the Lipschitz constant. We further propose a fully parameter-free cubic regularization (FF-CR) algorithm that does not require any parameters of the problem, including the Lipschitz constant and the upper bound of the distance from the initial point to the optimal solution. We also prove that the iteration complexity of the FF-CR algorithm to obtain an $\epsilon$-optimal solution with respect to the gradient norm is upper bounded by $\mathcal{O}(\rho^{2/3}\|z_0-z^*\|^{4/3}\epsilon^{-2/3}) $. Numerical experiments show the efficiency of both algorithms. To the best of our knowledge, the proposed FF-CR algorithm is a completely parameter-free second-order algorithm, and its iteration complexity is currently the best in terms of $\epsilon$  under the termination criterion of the gradient norm.
\end{abstract}


\section{Introduction}
In this paper, we consider the following unconstrained minimax optimization problem:
\begin{equation}\label{P}
\min_{x\in\mathbb{R}^{m}} \max_{y\in\mathbb{R}^{n}} f(x,y),
\end{equation}where $f:\mathbb{R}^{m} \times \mathbb{R}^{n} \to \mathbb{R} $ is a continuously differentiable function. Throughout this paper, we assume that the function $f(x, y)$ is convex in $x$ for all $y\in\mathbb{R}^{n}$ and concave in $y$ for all $x\in\mathbb{R}^{m}$.

Minimax optimization problems have a wide range of applications in game theory \citep{Neumann1953TheoryOG,Blackwell1979TheoryOG}, robust optimization \citep{Ben-Tal2009RobustOP,Gao2016DistributionallyRS}, and many other fields.
Recently, new applications in the field of machine learning and data science, such as generative adversarial networks (GANs) \citep{Goodfellow2014GenerativeAN,Arjovsky2017WassersteinGA}, adversarial learning \citep{Sinha2018CertifiableDR}, AUC maximization \citep{Hanley1982TheMA,Ying2016StochasticOA}, have further stimulated the research interest of many international scholars in the minimax problems.

Depending on the available information of the objective function, there are three main types of optimization algorithms to solve the minimax problem \eqref{P}, including zeroth-order, first-order, and second-order optimization algorithms, which use the function value, gradient, and Hessian information of the objective function, respectively. Compared with zeroth-order and first-order optimization algorithms, second-order optimization algorithms have attracted much attention due to their faster convergence speed. In this paper, we focus on second-order optimization algorithms for solving \eqref{P}.

Currently, there are relatively few studies on second-order algorithms for solving minimax optimization problems. Existing second-order algorithms can be divided into two categories, i.e., implicit algorithms \citep{Bullins2022HigherMF,Ostroukhov2020TensorMF,Monteiro2012IterationOA,Jiang2022GeneralizedOM} and explicit algorithms \citep{Huang2022CubicRN,Huang2022AnAR,Jiang2024AdaptiveAO,Lin2022ExplicitSM,Nesterov2006bCubicRO}.

For the implicit case, these methods compute the subproblem of the current iteration involving update iteration, leading to implicit update rules.  \citet{Monteiro2012IterationOA} proposed a Newton proximal extragradient method for monotone variational inequality problems, which achieves an iteration complexity of $\mathcal{O}(\epsilon^{-2/3})$.  
To achieve faster convergence rates, methods that exploit $p$-order derivatives of $f(x,y)$ have been proposed, such as the higher-order generalized optimistic algorithm \citep{Jiang2022GeneralizedOM}, the higher-order mirror-prox algorithm \citep{Bullins2022HigherMF}, and the restarted higher-order mirror-prox algorithms \citep{Ostroukhov2020TensorMF}.
All of them can achieve an iteration complexity of $\mathcal{O}(\epsilon^{-2/(p+1)})$ in terms of the $\epsilon$-optimal solution of the primal-dual gap. 
In addition, \citet{Ostroukhov2020TensorMF} proved that the iteration complexity of the restarted higher-order mirror-prox algorithm is $\mathcal{O}(\epsilon^{-2/(p+1)}\log(\epsilon^{-1}))$ in terms of the $\epsilon$-optimal solution of the gradient norm.

For the explicit case, \citet{Huang2022CubicRN} proposed a homotopy-continuation cubic regularized Newton algorithm for \eqref{P} with Lipschitz continuous gradients and Hessians, which achieves an iteration complexity of $\mathcal{O}(\log(\epsilon^{-1}))$ (resp. $\mathcal{O}(\epsilon^{-(1-\theta)\theta^2})$ with  $ \theta \in (0,1)$) under a Lipschitz-type (resp. H{\"o}lderian-type) error bound condition in terms of the $\epsilon$-optimal solution of $\|z_k-z^*\|$. 
For problems with only Lipschitz continuous Hessians, \citet{Lin2022ExplicitSM} proposed several Newton-type algorithms which all achieve an iteration complexity of $\mathcal{O}(\rho^{2/3}\|z_0-z^*\|^2\epsilon^{-2/3})$ in terms of the $\epsilon$-optimal solution of the restricted primal-dual gap. 
For more general monotone variational inequalities problems, \citet{Nesterov2006bCubicRO} proposed a dual Newton's method with a complexity bound  $\mathcal{O}(\epsilon^{-1})$.  \citet{Huang2022AnAR} proposed an approximation-based regularized extra-gradient scheme, which exploits $p$-order derivatives of $f(x,y)$ and achieves an iteration complexity of $\mathcal{O}(\epsilon^{-2/(p+1)})$. All the above explicit methods need to solve cubic or higher-order regularized subproblems in each iteration.

It should be noted that for most of the existing results mentioned above, achieving optimal
complexity of the algorithm requires the assumption of knowing precise information about some parameters of the problem, such as the Lipschitz constant, the upper bound of the distance from the initial point to the optimal solution, etc. Accurately estimating these parameters is often challenging, and conservative estimates can significantly impact algorithm performance \citep{Lan2023OptimalAP}. Therefore, designing algorithms with complexity guarantees without relying on inputs of these parameters has attracted considerable attention recently.
This type of method is called parameter-free algorithm.

There are very few existing research results on parameter-free second-order methods for solving minimax problems \eqref{P}. 
\cite{Jiang2022GeneralizedOM} proposed an optimistic second-order method (OSOM), where a line search strategy is employed to remove the need for prior knowledge of the Lipschitz constant. It achieves an iteration complexity of $\mathcal{O}(\rho^{2/3}\|z_0-z^*\|^2\epsilon^{-2/3})$  in terms of the $\epsilon$-optimal solution of the restricted primal-dual gap.
\cite{Liu2022RegularizedNM} proposed a $\nu$-regularized extra-Newton method (ReNewton) for solving monotone variational inequality problems with an iteration complexity of $\mathcal{O}(\epsilon^{-2/(2+\nu)})$ when $\nu$ is known, and proposed an universal regularized extra-Newton method with an iteration complexity of $\mathcal{O}(\epsilon^{-4/(3+3\nu)})$ when $\nu$ is unknown, where $\nu$ is the H{\"o}lder continuous constant. The proposed methods also employ a line search strategy to eliminate the need for problem parameters.
Recently, \cite{Jiang2024AdaptiveAO} 
proposed an adaptive second-order optimistic  method (ASOM), which solves a quadratic regularization subproblem at each iteration. It achieves an iteration complexity of $\mathcal{O}(\rho^{2/3}\|z_0-z^*\|^2\epsilon^{-2/3})$ (resp. $\mathcal{O}(\rho\|z_0-z^*\|^2\epsilon^{-1})$) in terms of the $\epsilon$-optimal solution of the restricted primal-dual gap (resp. gradient norm) under the assumption that the Hessian is Lipschitz continuous when the Lipschitz constant is known. If the Lipschitz constant is not known, the algorithm then requires an additional assumption that the gradient is Lipschitz continuous and can prove an iteration complexity of $\mathcal{O}\left((\ell^{4/3}\rho^{2/3}\|z_0-z^*\|^{2/3} + \rho^2 \|z_0-z^*\|^2 )\epsilon^{-2/3}\right)$ $\left(\text{resp.}~ \mathcal{O}\left((\ell\rho\|z_0-z^*\| + \rho^2 \|z_0-z^*\|^2 )\epsilon^{-1}\right)\right)$  in terms of the $\epsilon$-optimal solution of the restricted primal-dual gap (resp. gradient norm).
 

To the best of our knowledge,  no parameter-free second-order algorithm has previously achieved the complexity bound $\mathcal{O}(\epsilon^{-2/3})$ with respect to the gradient norm for solving unconstrained convex-concave minimax problems. In this paper, we propose a completely parameter-free second-order algorithm for solving \eqref{P} that achieves this complexity without any prior knowledge of the parameters.

\subsection{Other Related Works}

There are many existing research results on first-order optimization algorithms for solving convex-concave minimax optimization problems. For instance, \cite{Korpelevich1976TheEM} proposed an extra-gradient (EG) algorithm, and its linear convergence rate for solving smooth strongly convex-strongly concave minimax problems and bilinear minimax problems was proved in \cite{Tseng1995OnLC}.
\cite{Nemirovski2004ProxWR} proposed a mirror-prox method, which achieves an iteration complexity of $\mathcal{O}(\ell\epsilon^{-1})$ in terms of the primal-dual gap for constrained minimax problems, where $\ell$ is the gradient Lipschitz continuous constant of $f(\cdot,\cdot)$.
\cite{Nesterov2007DualEA} proposed a dual extrapolation algorithm, which has the same iteration complexity.
\cite{Mokhtari2020AUA} (\cite{Mokhtari2020ConvergenceAEb}) established an overall complexity of $\mathcal{O}({\color{blue}\kappa}\log\epsilon^{-1})$ (resp. $\mathcal{O}(\ell\epsilon^{-1})$) for bilinear and strongly convex-strongly concave minimax problems (resp. smooth convex-concave minimax problems) for both the EG and the optimistic gradient descent ascent (OGDA) method \citep{Popov1980AMO} , where $\kappa$ is the condition number of $f(\cdot,\cdot)$. For more related results, we refer to \citep{Hsieh2019ONTC,Solodov1999AHA,Kotsalis2022SimpleAO,Ouyang2021LowerCB}
and the references therein.

For nonconvex-strongly concave minimax problems, various first-order algorithms have been proposed in recent works \citep{Bot,Jin,Lin2019,Lin2020,Lu}, and all of them can achieve the iteration complexity of  $\tilde{\mathcal{O}} (\ell\kappa_y^2\epsilon ^{-2} )$ in terms of stationary point of $\Phi (\cdot) = \max_{y\in \mathcal{Y}} f(\cdot, y)$ (when $x \in \mathbb{R}^{m}$), or stationary point of $f(x,y)$,  where $\kappa_y$ is the condition number for $f(x,\cdot)$. Furthermore, \cite{Zhang2021} proposed a generic acceleration framework which can improve the iteration complexity to $\tilde{\mathcal{O}}(\ell\sqrt{\kappa_y}\epsilon ^{-2} )$. Some second-order algorithms have also been proposed in recent works \citep{Luo2022FindingSS,Chen2023ACR,Yao2024TwoTR} to find second-order stationary point.

For general nonconvex-concave minimax problems, there are two types of algorithms, i.e., multi-loop algorithms and single-loop algorithms. For example, various multi-loop algorithms have been proposed in \citep{Kong,Lin2020,Nouiehed,Ostro,Rafique,Thek2019}. The best-known iteration complexity of multi-loop algorithms for solving nonconvex-concave minimax problems is $\tilde{\mathcal{O}}(\ell^{1.5} \epsilon ^{-2.5})$, which was achieved by \cite{Lin2020}.  
Various single-loop algorithms have also been proposed,  e.g., the gradient descent-ascent (GDA) algorithm \citep{Lin2019}, the hybrid block successive approximation (HiBSA) algorithm \citep{Lu}, the unified single-loop alternating gradient projection (AGP) algorithm \citep{Xu}, and the smoothed GDA algorithm \citep{Zhang}. Both the AGP algorithm and the smoothed GDA algorithm achieve the best-known iteration complexity, i.e., $\mathcal{O}( \ell^4\epsilon ^{-4})$, among single-loop algorithms for solving nonconvex-concave minimax problems. For solving nonconvex-linear minimax problems, \cite{Pan} proposed a new alternating gradient projection algorithm, which finds an $\epsilon$-first-order Nash equilibrium point \citep{Nouiehed} within $\mathcal{O}( \ell^2\epsilon ^{-3})$ iterations.

Next, we briefly review some existing works that focus on zeroth-order algorithms for solving the minimax optimization problem. For instance,
\cite{Wang} proposed a zeroth-order gradient descent ascent (ZO-GDA) algorithm and a zeroth-order gradient descent multi-step ascent (ZO-GDMSA) algorithm for the nonconvex-strongly concave setting, and the total complexity to find an $\epsilon$-stationary point of $\Phi (\cdot) = \max_{y\in \mathcal{Y}} f(\cdot, y)$ is $\mathcal{O}(\kappa_y^{5}(m+n) \epsilon^{-2})$ and $\mathcal{O}(\kappa_y(n+m\kappa_y\log(\epsilon^{-1}))\epsilon^{-2})$, respectively. \cite{xu21zeroth} proposed a zeroth-order alternating stochastic gradient projection (ZO-AGP) algorithm to solve the nonconvex-concave minimax problem, and proved that the iteration complexity of the ZO-AGP algorithm is $\mathcal{O}(\ell^4\epsilon^{-4})$ with the number of function value estimates per iteration being bounded by $\mathcal{O}(n+m)$.
\cite{xu22zeroth} proposed a zeroth-order alternating gradient descent ascent (ZO-AGDA)
algorithm and a zeroth-order variance reduced alternating gradient descent ascent (ZO-VRAGDA) algorithm for solving a class of nonconvex-nonconcave minimax problems under the deterministic and the stochastic setting, respectively. The iteration complexity of the ZO-AGDA algorithm and the ZO-VRAGDA algorithm have been proved to be $\mathcal{O}(\ell(n+m\kappa_y^2)\epsilon^{-2})$ and $\mathcal{O}(\ell\kappa_y^4nm\epsilon^{-3})$, respectively. For more related results, we refer to \citep{Liu,Beznosikov,Sadiev,Xu20,Shen} and the references therein.

\subsection{Contributions}
In this paper, we focus on developing a fully parameter-free second-order algorithm for solving unconstrained convex-concave minimax optimization problems. We summarize our contributions as follows.

We propose a Lipschitz-free cubic regularization (LF-CR) algorithm for solving convex-concave minimax optimization problems without knowledge of the Lipschitz constant. It can be proved that the iteration complexity  of the LF-CR algorithm to obtain an $\epsilon$-optimal solution with respect to the restricted primal-dual gap (resp. gradient norm) is upper bounded by  $\mathcal{O}(\rho^{2/3}\|z_0-z^*\|^2\epsilon^{-2/3})$ (resp. $\mathcal{O}(  \rho \|z_0-z^*\|^2 \epsilon^{-1})$), where $z_0=(x^0,y^0)$ is a pair of initial points, $z^*=(x^*,y^*)$ is a pair of the optimal solutions, and $\rho$ is the Lipschitz constant. This complexity bound is consistent with existing second-order methods in \cite{Lin2022ExplicitSM}. Compared to the adaptive second-order optimistic method (ASOM) in \cite{Jiang2024AdaptiveAO},  which has an iteration complexity of $\mathcal{O}\left((\ell^{4/3}\rho^{2/3}\|z_0-z^*\|^{2/3} + \rho^2 \|z_0-z^*\|^2 )\epsilon^{-2/3}\right)$ $\left(\text{resp.}~ \mathcal{O}\left((\ell\rho\|z_0-z^*\| + \rho^2 \|z_0-z^*\|^2 )\epsilon^{-1}\right)\right)$  in terms of the $\epsilon$-optimal solution of the restricted primal-dual gap (resp. gradient norm),
the LF-CR algorithm has better iteration complexity bound with respect to the Lipschitz constant $\rho$.

We further propose a fully parameter-free cubic regularization (FF-CR) algorithm without any knowledge of the problem parameters, including the Lipschitz constants and the upper bound of the distance from the initial point to an optimal solution. 
The iteration complexity to obtain an $\epsilon$-optimal solution with respect to  the gradient norm can be proved to be upper bounded by $\mathcal{O}\big( \rho^{2/3}\|z_0-z^*\|^{4/3}\epsilon^{-2/3}  \big) $. Numerical experiments show the efficiency of the two proposed algorithms. 

To the best of our knowledge, the proposed FF-CR algorithm is a completely parameter-free second-order algorithm, and its iteration complexity is currently the best in terms of $\epsilon$  under the termination criterion of the gradient norm.
Furthermore, our method only requires  a mild assumption that the Hessian is Lipschitz continuous and the norm of Hessian matrix at the the initial point is upper bounded, without making additional assumptions of error bound conditions as in \citep{Huang2022CubicRN}.
Compared to the restarted higher-order mirror-prox algorithm with the knowledge of parameters \citep{Ostroukhov2020TensorMF}, we improve the complexity for the convex-concave setting by a logarithmic factor. Moreover, our method achieves better iteration complexity compared to the adaptive second-order optimistic method (ASOM) in \citep{Jiang2024AdaptiveAO}, which has an iteration complexity of $\mathcal{O}\left((\ell\rho\|z_0-z^*\| + \rho^2 \|z_0-z^*\|^2 )\epsilon^{-1}\right)$ in terms of gradient norm. 
Table \ref{complexity} provides a detailed comparison of the two proposed algorithms with existing parameter-free second-order algorithms for solving convex-concave minimax problems.

\begin{table}[!ht]
	\centering
	\caption{Comparison of parameter-free second-order algorithms for solving convex-concave minimax problems}\label{complexity}
		\resizebox{\textwidth}{!}{\begin{threeparttable}
			\centering
			\begin{tabular}{c|c|c}
				\hline
				Algorithms  & Termination Criterion \tnote{1} & Iteration Complexity\tnote{2}  \\ \hline
				ReNewton \citep{Liu2022RegularizedNM}  & $\textsc{gap}(x, y; \beta) \leqslant \epsilon$ & $\mathcal{O}(\rho^{2/3}D^2\epsilon^{-2/3} )$ \\ 
				OSOM \citep{Jiang2022GeneralizedOM}  & $\textsc{gap}(x, y; \beta) \leqslant \epsilon$ & $\mathcal{O}(\rho^{2/3}\|z_0-z^*\|^2\epsilon^{-2/3} )$ \\ 
				ASOM \citep{Jiang2024AdaptiveAO}   & $\textsc{gap}(x, y; \beta) \leqslant \epsilon$ & $\mathcal{O}\left((\ell^{4/3}\rho^{2/3}\|z_0-z^*\|^{2/3} + \rho^2 \|z_0-z^*\|^2 )\epsilon^{-2/3}\right)$ \\
			 & $\|\nabla f(x, y)\| \leqslant \epsilon$ & $\mathcal{O}\left((\ell\rho\|z_0-z^*\| + \rho^2 \|z_0-z^*\|^2 )\epsilon^{-1}\right)$ \\
				{\bfseries LF-CR (This paper)}  & $\textsc{gap}(x, y; \beta) \leqslant \epsilon$ & $\mathcal{O}(\rho^{2/3}\|z_0-z^*\|^2\epsilon^{-2/3})$ \\
			    & $\|\nabla f(x, y)\| \leqslant \epsilon$ & $\mathcal{O}(\rho\|z_0-z^*\|^2\epsilon^{-1})$
				\\
				{\bfseries FF-CR (This paper)}  & $\|\nabla f(x, y)\| \leqslant \epsilon$ & 				
				$ \mathcal{O}\big( \rho^{2/3}\|z_0-z^*\|^{4/3}\epsilon^{-2/3} \big) $
				\\ \hline
			\end{tabular}
			\begin{tablenotes}
			\footnotesize
				\item [1] ``$\textsc{gap}(x, y; \beta)$'' denotes the (restricted) primal-dual gap defined in \eqref{def:reGAP:1}.
				\item [2] ``$D$'' denotes the diameter of the constraint set, ``$\ell$" and ``$\rho$" represent the Lipschitz constants of the gradient and the Hessian, respectively.
			\end{tablenotes}
	\end{threeparttable}}
\end{table}


\subsection{Organization}
The rest of this paper is organized as follows.
In Section 2, we propose a Lipschitz-free cubic regularization (LF-CR) algorithm and prove its iteration complexity. In Section 3, we propose a fully parameter-free cubic regularization (FF-CR) algorithm and prove its iteration complexity. In Section 4, we conduct experiments to demonstrate the efficiency of our algorithm. Finally, some conclusions are made in Section 5.

{\bfseries Notation}.
We use $\|\cdot\|$ to denote the spectral norm of matrices and Euclidean norm of vectors. We use $I_n \in \mathbb{R}^{n\times n}$ to denote identity matrix.
For a function $f(x,y):\mathbb{R}^m\times\mathbb{R}^n\rightarrow \mathbb{R}$, we use $\nabla_{x} f(x,y)$ (or $\nabla_{y} f(x, y)$) to denote the partial gradient of $f$ with respect to the first variable (or the second variable) at point $(x, y)$. Additionally, we use $\nabla f(x, y)=(\nabla_{x}f(x,y), \nabla_{y}f(x,y))$ to denote the full gradient of $f$ at point $(x, y)$. Similarly, we denote $\nabla_{xx}^2 f(x, y)$, $ \nabla^2_{xy} f(x, y)$, $\nabla_{yx}^2f(x,y)$ and $ \nabla_{yy}^2 f(x, y)$ as the second-order partial derivatives of $f(x,y)$ and use $\nabla^2f(x,y)=\begin{bmatrix} \nabla_{xx}^{2} f(x, y) & \nabla_{xy}^{2} f(x, y) \\ \nabla_{yx}^{2} f(x, y) & \nabla_{yy}^{2} f(x, y) \end{bmatrix}$ to denote the full Hessian of $f$ at point $(x, y)$. We use the notation $\mathcal{O} (\cdot), \Omega(\cdot)$ to hide only absolute constants which do not depend on any problem parameter, and $\tilde{\mathcal{O}} (\cdot) $ to hide only absolute constants and logarithmic factors. 
Finally, we define the vector $z = [x;y] \in \mathbb{R}^{m+n}$ and the gradient operator $F: \mathbb{R}^{m+n} \to \mathbb{R}^{m + n}$ as follows, 
\begin{equation}\label{def:operator}
F(z) = \begin{bmatrix} \nabla_{x} f(x, y) ; - \nabla_{y} f(x, y) \end{bmatrix}. 
\end{equation}
Accordingly, the Jacobian of $F$ is defined as follows, 
\begin{equation}\label{def:Jacobian}
D F(z) = \begin{bmatrix} ~~ \nabla_{xx}^{2} f(x, y) & ~~\nabla_{xy}^{2} f(x, y) \\ -\nabla_{yx}^{2} f(x, y) & -\nabla_{yy}^{2} f(x, y) \end{bmatrix} \in \mathbb{R}^{(m + n) \times (m + n)}. 
\end{equation}

\section{A Lipschitz-Free Cubic Regularization (LF-CR) Algorithm}\label{Adaptive-Newton-MinMax}

In this section, we propose a Lipschitz-free cubic regularization (LF-CR) algorithm that does not require prior knowledge of the Lipschitz constant. We iteratively approximate the Lipschitz constants and show that the iteration complexity of the proposed algorithm remains the same order in the absence of information about the Lipschitz constants. Moreover, the LF-CR algorithm is a key subroutine in our second parameter-free algorithm that we will present in the next section.

The proposed LF-CR algorithm is based on the Newton-MinMax algorithm \citep{Lin2022ExplicitSM}, which consists of the following two important algorithmic components at each iteration:
\begin{itemize}
	\item Gradient update through cubic regularization minimization: for given $\hat{z}_k$, compute $z_{k+1}$ such that it is a solution of the following nonlinear equation:
	\begin{align}\label{gradient update}
	F(\hat{z}_k)+ DF(\hat{z}_k)(z-\hat{z}_k)+6\rho\|z-\hat{z}_k\|(z-\hat{z}_k)=0,
	\end{align}
	where $\rho$ is the Lipschitz constant of the Hessian of $f(x,y)$.
	\item Extragradient update:
	\begin{align*}
	\hat{z}_{k+1}=\hat{z}_k-\lambda_{k+1}F(z_{k+1}),
	\end{align*}
    where $\lambda_{k+1}=\frac{c_k}{\rho\|z_{k+1}-\hat{z}_k\|}$ for some constant $c_k$.  Note that $\|z_{k+1}-\hat{z}_k\|=0$ if and only if $\|F(\hat{z}_k)\|=0,$ which will be proved later.
\end{itemize}
Although the Newton-MinMax algorithm achieves the best known iteration complexity,  it requires  knowing $\rho$, the Lipschitz constant of the Hessian of $f(x,y)$, which limits its applicability.
To eliminate the need for the prior knowledge of $\rho$, we propose a LF-CR algorithm that employs a line search strategy in the first gradient update step. Similar line search  strategies have been used in \citep{Liu2022RegularizedNM,Nesterov2006CubicRO,Nesterov2015UniversaGM,Nesterov2018LecturesOC}.  More detailedly, in each iteration of the proposed LF-CR algorithm, we aim to search a pair  $(H_{k,i_k},z_{k,i_k})$ by backtracking such that
\begin{align}\label{backtracking condition}
\|F(z_{k,i_k}) - F(\hat{z}_k)- DF(\hat{z}_k)(z_{k,i_k}-\hat{z}_k)\| \leqslant \frac{H_{k,i_k}}{2}\|z_{k,i_k}-\hat{z}_k\|^2,
\end{align}
where $z_{k,i_k}$ is a solution of the following nonlinear equation:
\begin{align}\label{nesubproblem}
F(\hat{z}_k)+ DF(\hat{z}_k)(z-\hat{z}_k)+6H_{k,i_k}\|z-\hat{z}_k\|(z-\hat{z}_k)=0.
\end{align}
We then introduce how to solve the nonlinear equations in \eqref{nesubproblem}. Similar to the discussion in \citep{Lin2022ExplicitSM}, \eqref{nesubproblem} can be reformulated as finding a pair of $(\theta_{k,i_k},z_{k,i_k})$ such that
\begin{align*}
\theta_{k,i_k} &= 6H_{k,i_k}\|z_{k,i_k}-\hat{z}_k\|,\\
(DF(\hat{z}_k)&+\theta_{k,i_k} I_{m+n})(z_{k,i_k}-\hat{z}_k)=-F(\hat{z}_k).
\end{align*}
Solving the above nonlinear system of equations is equivalent to first solving the following univariate nonlinear equation with respect to $\theta$:
\begin{align}
&\phi(\theta):=\sqrt{\|(DF(\hat{z}_k)+\theta I_{m+n})^{-1}F(\hat{z}_k)\|^2}-\frac{\theta}{6H_{k,i_k}}=0. \label{ne1}
\end{align}
Similar to the proof of Proposition 4.5 in \citep{Lin2022ExplicitSM}, when $\theta > 0$, $\phi(\theta)$ is a strictly decreasing and convex function. Therefore, we can first use Newton's method to solve the one-dimensional nonlinear equation \eqref{ne1} to find the optimal solution of $\theta$. Then by substituting it into the following equation and solving the following linear system, 
\begin{align}
&(DF(\hat{z}_k)+\theta I_{m+n})(z_{k,i}-\hat{z}_k)=-F(\hat{z}_k)\label{ne2},
\end{align}
we obtain $z_{k,i_k}$.   
Next, the proposed LF-CR algorithm performs an extra gradient update on $z_{k,i_k}$. The detailed algorithm is formally described in Algorithm \ref{alg:1}.

\begin{algorithm}[th!]
	\caption{A Lipschitz-free cubic regularization (LF-CR) algorithm }
	\label{alg:1}
	\begin{algorithmic}
		\STATE{\textbf{Step 1}: {Input $z_0,\tilde{z}_0$; Set $H_0=\frac{\|DF(z_0)-DF(\tilde{z}_0)\|}{\|z_0-\tilde{z}_0\|} (z_0 \neq \tilde{z}_0)$, $k=1$, $\hat{z}_1=z_1=z_0.$}}
		\STATE{\textbf{Step 2}: Find a pair $(H_{k,i},z_{k,i})$} to satisfy  \eqref{backtracking condition}:
		\STATE{\quad\textbf{(a)}: Set $i=1$, $H_{k,i}=H_{k-1}$.}
		\STATE{\quad\textbf{(b)}: Update $z_{k,i}$ such that it is a solution of the following nonlinear equation:	
			\begin{equation}\label{step2b}
				F(\hat{z}_k)+ DF(\hat{z}_k)(z-\hat{z}_k)+ 6H_{k,i}\|z-\hat{z}_k\|(z-\hat{z}_k)=0.
		\end{equation}}
		\STATE{\quad\textbf{(c)}: If 
			\begin{equation}\label{step2c}
				\|F(z_{k,i}) - F(\hat{z}_k)- DF(\hat{z}_k)(z_{k,i}-\hat{z}_k)\| \leqslant \frac{H_{k,i}}{2}\|z_{k,i}-\hat{z}_k\|^2,
			\end{equation}
			\qquad\quad~$H_{k}=H_{k,i}$,  $z_{k+1}=z_{k,i}$; otherwise, set $H_{k,i+1}=2H_{k,i}$, $i=i+1, $ go to Step 2(b).}	
		\STATE{\textbf{Step 3}: If $\|z_{k+1}-\hat{z}_k\|=0$, stop and output $(\hat{z}_k,H_k)$.}
		\STATE{\textbf{Step 4}: Update $\hat{z}_k$: 
			\begin{equation}\label{step4}
				\hat{z}_{k+1}=\hat{z}_k-\lambda_{k+1}F(z_{k+1}),
		\end{equation} \qquad \quad~with $\lambda_{k+1}=\frac{c_k}{H_k\|z_{k+1}-\hat{z}_k\|}$.}
		\STATE{\textbf{Step 5}: If 
			some stationary condition or stopping criterion is satisfied,  let $\bar{z}_k=\frac{1}{\sum_{i=1}^{k}\lambda_{i+1}}\sum_{i=1}^{k}\lambda_{i+1} z_{i+1}$ and output $(\bar{z}_k,H_{k})$;
			otherwise, set $k=k+1, $ go to Step 2(a).}
	\end{algorithmic}
\end{algorithm}


In the following subsection, we establish the iteration complexity of the LF-CR algorithm to obtain an $\epsilon$-optimal solution of \eqref{P} with respect to  the restricted primal-dual gap and gradient norm. 
Our complexity analysis builds upon the framework of \citep{Lin2022ExplicitSM} while addressing the more challenging scenario where the Hessian Lipschitz constant is unknown. We relax the assumption and only require that the Hessian matrix is locally Lipschitz continuous. The key difference in the proof is that we establish a lower bound for the adaptive step size $\lambda_{k+1}=\frac{c_k}{H_k\|z_{k+1}-\hat{z}_k\|}$, where $H_k$ is estimated dynamically via line search rather than being a fixed constant. This requires constraining $H_k$ in the algorithm design and carefully considering the overhead of line search in the complexity analysis. In addition, under the assumption of local Lipschitz continuity, we need to prove the boundedness of the sequence of the iteration points.

\subsection{Complexity analysis}
Before we prove the iteration complexity of the LF-CR algorithm, we first define the optimal solution of \eqref{P} denoted by $(x^*,y^*)$ and the stationarity gap function, which provides a performance measure for how close a point $(\hat{x}, \hat{y})$ is to $(x^*,y^*)$.
\begin{defi}\label{def:optsol}
	A pair $(x^*,y^*)$ is called an optimal solution of problem \eqref{P}, if
	\begin{align*}
		f(x^*,y) \leqslant f(x^*,y^*) \leqslant f(x,y^*) ~\text{for all}~ x \in \mathbb{R}^m ~\text{and}~ y \in \mathbb{R}^n.
	\end{align*} 
\end{defi}
Similar to \citep{Lin2022ExplicitSM}, we define the restricted primal-dual gap function for \eqref{P} as follows.
\begin{defi}\label{def:reGAP}
	Denote
	\begin{equation}
	\textsc{gap}(\hat{x}, \hat{y}; \beta) = \max_{y \in \mathbb{B}_\beta(y^\star)} f(\hat{x}, y) - \min_{x \in \mathbb{B}_\beta(x^\star)} f(x, \hat{y}),\label{def:reGAP:1}
	\end{equation}
	where $(x^*,y^*)$ is an optimal solution of $f(x,y)$ and $\beta$ is large enough so that for a given $(\hat{x},\hat{y})$, $\|\hat{z} - z^\star\|\leqslant \beta$, $\mathbb{B}_\beta(x^*)=\{ x : \|x-x^*\| \leqslant \beta \}$  and $\mathbb{B}_\beta(y^*)=\{ y : \|y-y^*\| \leqslant \beta \}$.
\end{defi}
\begin{defi}\label{es-reGAP}
	A pair $(\hat{x}, \hat{y})$ is called an $\epsilon$-optimal solution of \eqref{P} with respect to the restricted primal-dual gap if $\textsc{gap}(\hat{x}, \hat{y}; \beta) \leqslant \epsilon.$
\end{defi}
\begin{defi}\label{es-gn}
     A point $(\hat{x}, \hat{y})$ is called an $\epsilon$-optimal solution of \eqref{P} with respect
	to the gradient norm if $\|\nabla f(\hat{x}, \hat{y})\| \leqslant \epsilon.$
\end{defi}
Similar to the definition of gradient Lipschitz continuity in \citep{Vladarean2021A}, we introduce the following definitions for local and global Hessian Lipschitz continuity.
\begin{defi}
	A function $f(x, y)$ is globally Hessian Lipschitz continuous, if there exists a  positive scalar $\rho$ such that 
	\begin{align*}
		\| \nabla^2 f(x, y) - \nabla^2 f(x', y')\| \leqslant \rho\|(  x - x', y - y' )\|, ~\forall (x, y),(x', y') \in \mathbb{R}^{m} \times \mathbb{R}^{n}.
	\end{align*} 
\end{defi}
\begin{defi}
	A function $f(x, y)$ is locally Hessian Lipschitz continuous, if there exists a positive scalar $\rho({\mathcal{C}})$ such that for any compact subset $\mathcal{C} \in \mathbb{R}^{m} \times \mathbb{R}^{n}$,
	\begin{align*}
		\| \nabla^2 f(x, y) - \nabla^2 f(x', y')\| \leqslant \rho({\mathcal{C}})\|(  x - x', y - y' )\|, ~\forall (x, y),(x', y') \in \mathcal{C}.
	\end{align*} 
\end{defi}
It is important to note that the local Hessian-Lipschitz condition is strictly weaker than the global condition, and therefore covers a wider range of functions. For example, functions like $e^{x}$ and $x^p ~(p \geqslant 4)$ are locally Hessian-Lipschitz continuous, but not globally.

We also need to make the following assumptions about the function $f(x,y).$
\begin{ass}\label{ass:func}
The function $f(x, y)$ is convex-concave, i.e., $f(\cdot, y)$ is convex for any fixed $y \in \mathbb{R}^n$ and
	$f(x, \cdot)$ is concave for any fixed $x \in \mathbb{R}^m$.
	The optimal solution of problem \eqref{P} exists.
\end{ass}
\begin{ass}\label{ass:smoothl}
	The function $f(x, y)$ is locally Hessian Lipschitz continuous.
\end{ass} 
\begin{lem}\label{2:add1}
		Suppose that Assumptions \ref{ass:func} and \ref{ass:smoothl} hold. Let $\{z_k\}$ and $\{\hat{z}_k\}$ be the sequences generated by Algorithm \ref{alg:1}.
		Then, $\|z_{k+1}-\hat{z}_k\|=0$ if and only if $\|F(\hat{z}_k)\|=0.$
\end{lem}
\begin{proof}
		By Step 2 in Algorithm \ref{alg:1}, we get 
		\begin{align}
			&F(\hat{z}_{k}) + DF(\hat{z}_{k})(z_{k+1}-\hat{z}_{k}) + 6H_k \|z_{k+1}-\hat{z}_{k}\|(z_{k+1}-\hat{z}_{k}) = 0, \label{2:add1:1}\\
			&\|F(z_{k+1}) - F(\hat{z}_{k})- DF(\hat{z}_{k})(z_{k+1}-\hat{z}_{k})\| \leqslant \frac{H_k}{2}\|z_{k+1}-\hat{z}_{k}\|^2. \label{2:add1:2}
		\end{align}
Then, on the one hand, if $\|z_{k+1}-\hat{z}_k\|=0$, then it is straightforward to prove $\|F(\hat{z}_k)\|=0$. On the other hand, if $\|F(\hat{z}_k)\|=0$, by taking the inner product of both sides of \eqref{2:add1:1} with $z_{k+1}-\hat{z}_{k}$, we obtain
\begin{equation*}
	\left\langle z_{k+1}-\hat{z}_{k}, DF(\hat{z}_{k})(z_{k+1}-\hat{z}_{k}) \right\rangle + 6H_k \|z_{k+1}-\hat{z}_{k}\|^3 = 0,
\end{equation*}
which implies $\|z_{k+1}-\hat{z}_k\|=0$. The proof is completed.		
\end{proof}

The following lemma helps us prove the final convergence rate, it describes the iterations involved in the algorithm and provides us with a key descent inequality.

\begin{lem}\label{2:lem1}
	Suppose that Assumptions \ref{ass:func} and \ref{ass:smoothl} hold. Let $\{z_k\}$ and $\{\hat{z}_k\}$ be the sequences generated by Algorithm \ref{alg:1}. Assuming that $\|z_{k+1}-\hat{z}_k\| \neq 0$, set $\lambda_{k+1}=\frac{c_k}{H_k\|z_{k+1}-\hat{z}_k\|}$ with $c_k$ being a constant. If in addition $\frac{1}{33}\leqslant c_k \leqslant \frac{1}{13}$, then for any $z \in \mathbb{R}^{m}\times \mathbb{R}^{n}$ and $s \geqslant 1$, we have
	\begin{align}\label{2:lem1:1}
	&\sum_{k=1}^{s} \lambda_{k+1} \langle z_{k+1} - z, F(z_{k+1})\rangle \nonumber\\
	\leqslant&  - \frac{1}{2}\|\hat{z}_{s+1}-z_{0}\|^{2} +  \langle z_0 - z,z_0 - \hat{z}_{s+1}\rangle - \frac{1}{24} \sum_{k=1}^{s} \|z_{k+1} - \hat{z}_{k}\|^2.   
	\end{align}
\end{lem}

\begin{proof}
	By \eqref{step4} in Algorithm \ref{alg:1}, we have
	\begin{align}\label{2:lem1:2}
	&~\lambda_{k+1} \langle z_{k+1} - z, F(z_{k+1})\rangle \nonumber \\
	=&~ \lambda_{k+1} \langle z_{k+1} - \hat{z}_{k+1}, F(z_{k+1})\rangle+\lambda_{k+1} \langle \hat{z}_{k+1} - z_0, F(z_{k+1})\rangle +\lambda_{k+1} \langle z_0 - z, F(z_{k+1})\rangle \nonumber\\
	=&~ \lambda_{k+1} \langle z_{k+1} - \hat{z}_{k+1}, F(z_{k+1})\rangle +  \langle\hat{z}_{k+1} - z_0,\hat{z}_{k}-\hat{z}_{k+1}\rangle + \langle z_0 - z,\hat{z}_{k}-\hat{z}_{k+1}\rangle.
	\end{align}
	We first bound the first term of the right-hand side of \eqref{2:lem1:2}. 
	By \eqref{2:add1:1} and \eqref{2:add1:2}, we obtain 
	\begin{align}
	&~\langle z_{k+1} - \hat{z}_{k+1}, F(z_{k+1})\rangle \nonumber\\
	= &~\langle z_{k+1} - \hat{z}_{k+1}, F(z_{k+1})-F(\hat{z}_{k})-DF(\hat{z}_{k})(z_{k+1}-\hat{z}_{k})\rangle \nonumber\\
	&~+ \langle z_{k+1} - \hat{z}_{k+1},F(\hat{z}_{k})+DF(\hat{z}_{k})(z_{k+1}-\hat{z}_{k})\rangle \nonumber\\
	\leqslant &~ \frac{H_k}{2}\|z_{k+1} - \hat{z}_{k+1}\|\|z_{k+1}-\hat{z}_{k}\|^2 - 6H_k\|z_{k+1}-\hat{z}_{k}\|\langle z_{k+1}-\hat{z}_{k+1},z_{k+1}-\hat{z}_{k}\rangle. \label{lem2.1:2}
	\end{align}
	By the Cauchy-Schwarz inequality, we can easily get that
	\begin{align*}
	\|z_{k+1} - \hat{z}_{k+1}\|\|z_{k+1}-\hat{z}_{k}\|^2 
	&\leqslant \|z_{k+1}-\hat{z}_{k}\|^3+\|z_{k+1}-\hat{z}_{k}\|^2\|\hat{z}_{k}- \hat{z}_{k+1}\|,\\
	-\langle z_{k+1}-\hat{z}_{k+1},z_{k+1}-\hat{z}_{k}\rangle
	&\leqslant -\|z_{k+1}-\hat{z}_{k}\|^2 + \|z_{k+1}-\hat{z}_{k}\|\|\hat{z}_{k}-\hat{z}_{k+1}\|.
	\end{align*}
	By plugging the above two inequalities into \eqref{lem2.1:2}, and using $\frac{1}{33} \leqslant \lambda_{k+1}H_k\|z_{k+1}-\hat{z}_{k}\| \leqslant  \frac{1}{13}$, we have that
	\begin{align} \label{2:lem1:3}
	&~\lambda_{k+1}\langle z_{k+1} - \hat{z}_{k+1}, F(z_{k+1})\rangle \nonumber\\
	\leqslant & -\frac{11H_k\lambda_{k+1}}{2}\|z_{k+1}-\hat{z}_{k}\|^3 +  \frac{13H_k\lambda_{k+1}}{2}\|z_{k+1}-\hat{z}_{k}\|^2\|\hat{z}_{k} - \hat{z}_{k+1}\| \nonumber \\
	\leqslant & -\frac{1}{6}\|z_{k+1}-\hat{z}_{k}\|^2 +  \frac{1}{2}\|z_{k+1}-\hat{z}_{k}\|\|\hat{z}_{k} - \hat{z}_{k+1}\| \nonumber\\
	\leqslant & -\frac{1}{6}\|z_{k+1}-\hat{z}_{k}\|^2 + \frac{1}{8}\|z_{k+1}-\hat{z}_{k}\|^2 +  \frac{1}{2}\|\hat{z}_{k} - \hat{z}_{k+1}\|^2  \nonumber\\
	= & -\frac{1}{24}\|z_{k+1}-\hat{z}_{k}\|^2 +  \frac{1}{2}\|\hat{z}_{k} - \hat{z}_{k+1}\|^2, 
	\end{align}
	where the last inequality is by the Cauchy-Schwarz inequality. 
	
	On the other hand, it is easy to verify that 
	\begin{align}\label{2:lem1:4}
	\langle\hat{z}_{k+1} - z_0,\hat{z}_{k}-\hat{z}_{k+1}\rangle = \frac{1}{2}(\|\hat{z}_{k}-z_0\|^2 - \|\hat{z}_{k+1}-z_0\|^2 - \|\hat{z}_{k}-\hat{z}_{k+1}\|^2).
	\end{align}
	Plugging \eqref{2:lem1:3} and \eqref{2:lem1:4} into \eqref{2:lem1:2}, we have
	\begin{align}\label{2:lem1:5}
	&~\lambda_{k+1} \langle z_{k+1} - z, F(z_{k+1})\rangle \\
	\leqslant& -\frac{1}{24}\|z_{k+1}-\hat{z}_{k}\|^2  + \frac{1}{2}(\|\hat{z}_{k}-z_0\|^2 - \|\hat{z}_{k+1}-z_0\|^2 ) +  \langle z_0 - z,\hat{z}_{k}-\hat{z}_{k+1}\rangle. \nonumber
	\end{align}
	The proof is then completed by summing the above inequality for $k$ from $1$ to $s$ and using $\hat{z}_1 = z_0$.
\end{proof}

The following lemma  further gives bounds for the sequences $\{z_k\}$ and $\{\hat{z}_k\}$ generated by Algorithm \ref{alg:1}.

\begin{lem}\label{2:lem2}
		Suppose that Assumptions \ref{ass:func} and \ref{ass:smoothl} hold. Let $\{z_k\}$ and $\{\hat{z}_k\}$ be the sequences generated by Algorithm \ref{alg:1}. Assuming that $\|z_{k+1}-\hat{z}_k\| \neq 0 $, set $\lambda_{k+1}=\frac{c_k}{H_k\|z_{k+1}-\hat{z}_k\|}$ with $c_k$ being a constant. If in addition $\frac{1}{33}\leqslant c_k \leqslant \frac{1}{13}$, then for any $z \in \mathbb{R}^{m}\times \mathbb{R}^{n}$ and $s \geqslant 1$, we have
	\begin{align}
	\sum_{k=1}^{s} \lambda_{k+1} \langle z_{k+1} - z, F(z_{k+1})\rangle &\leqslant \frac{1}{2}\|z_0-z\|^2, \label{2:lem2:1}\\ 
	\sum_{k=1}^s  \|z_{k+1} - \hat{z}_{k}\|^2 &\leqslant 12 \|z_0-z^*\|^2, \label{2:lem2:2}\\
	\|z_{s+1}-z^*\|\leqslant 7\|z_0-z^*\|, &~\|\hat{z}_{s+1}-z^*\|\leqslant 3\|z_0-z^*\|, \label{2:lem2:3}
	\end{align}
	where $z^*=(x^*,y^*)$ is an optimal solution of \eqref{P}. 
\end{lem}

\begin{proof}
	By Lemma \ref{2:lem1} and  the Cauchy-Schwarz inequality,  we have
	\begin{align}\label{2:lem2:4}
	\sum_{k=1}^{s} \lambda_{k+1} \langle z_{k+1} - z, F(z_{k+1})\rangle &\leqslant \frac{1}{2}\|z_0-z\|^2 - \frac{1}{24} \sum_{k=1}^s \|z_{k+1} - \hat{z}_{k}\|^2.
	\end{align}
	Then, \eqref{2:lem2:1} holds directly by \eqref{2:lem2:4}. 
	Under the  convex-concave setting, by the optimal condition of $z^*$, we get
	\begin{equation}\label{2:lem2:5}
	\langle z_{k+1} - z^*, F(z_{k+1})\rangle \geqslant 0.
	\end{equation}
	Then, \eqref{2:lem2:2} holds by combining \eqref{2:lem2:4} with $z=z^*$ and \eqref{2:lem2:5}.  On the other hand, by setting $z=z^*$ in \eqref{2:lem2:4} and using \eqref{2:lem2:5},  and by the Cauchy-Schwarz inequality and Young's inequality, we have
	\begin{align*}
	0 &\leqslant~ \sum_{k=1}^{s} \lambda_{k+1} \langle z_{k+1} - z^*, F(z_{k+1})\rangle + \frac{1}{24} \sum_{k=1}^s \|z_{k+1} - \hat{z}_{k}\|^2 \\
	&\leqslant -\frac{1}{2}\|\hat{z}_{s+1}-z_{0}\|^{2} + \langle z_0 - \hat{z}_{s+1}, z_0 - z^*\rangle \\
	&\leqslant -\frac{1}{2}\|\hat{z}_{s+1}-z_{0}\|^{2} + \frac{1}{4}\| z_0 - \hat{z}_{s+1}\|^2 + \|z_0 - z^*\|^2,
	\end{align*}
	which implies that $\|\hat{z}_{s+1}-z_0\|  \leqslant 2\|z_0-z^*\|$ for any $s \geqslant 1$. Then, by \eqref{2:lem2:2} and the fact that $\hat{z}_1=z_0$, we further have 
	\begin{align*}
	\|z_{s+1}-z_0\| \leqslant \|z_{s+1}-\hat{z}_{s}\| + \|\hat{z}_{s}-z_0\| \leqslant 6 \|z_0-z^*\|.
	\end{align*}
	Therefore, by the Cauchy-Schwarz inequality, we have $\|\hat{z}_{s+1}-z^*\|\leqslant \|\hat{z}_{s+1}-z_0\| +  \|z_0-z^*\|  \leqslant 3\|z_0-z^*\|$ and $\|z_{s+1}-z^*\|  \leqslant \|z_{s+1}-z_0\| +  \|z_0-z^*\| \leqslant 7\|z_0-z^*\|$.
	The proof is then completed.
\end{proof}

We then provide a technical lemma that establishes a lower bound for $\sum_{k=1}^s \lambda_{k+1}$, which is key to proving the final iteration complexity.

\begin{lem}\label{2:lem3}
	Suppose that Assumptions \ref{ass:func} and \ref{ass:smoothl} hold. Let $\{z_{k,i}\}$, $\{z_k\}$ and $\{\hat{z}_k\}$ be the sequences generated by Algorithm \ref{alg:1}. Assuming that $\|z_{k+1}-\hat{z}_k\|\neq 0$, set $\lambda_{k+1}=\frac{c_k}{H_k\|z_{k+1}-\hat{z}_k\|}$ with $c_k$ being a constant. If in addition $\frac{1}{33}\leqslant c_k \leqslant \frac{1}{13}$, then for any  $s \geqslant 1$, we have $H_s\leqslant 2\rho$ and
	\begin{align}\label{2:lem3:1}
	\sum_{k=1}^s \lambda_{k+1} \geqslant \frac{s^{3/2}}{66\sqrt{3}H_{s}\|z_0 - z^*\|},
	\end{align}
	where $z^*=(x^*,y^*)$ is an optimal solution of \eqref{P} and $\rho$ is the local Hessian Lipschitz constant on the compact set $\mathcal{C}:= conv\{z^*, z_{k,i}, z_k, \hat{z}_k, i,k=1, 2,\cdots\}$, which is the convex hull generated by $z^*$ and the iterated point sequence.
\end{lem}
\begin{proof}
	By Assumption \ref{ass:func} and \eqref{step2b},  it is easy to see that for any $k,i \geqslant 1$,
	\begin{align*}
         6H_{k,i} \|z_{k,i}-\hat{z}_{k}\|^3 
         =& -\left\langle z_{k,i}-\hat{z}_{k}, DF(\hat{z}_{k})(z_{k,i}-\hat{z}_{k})\right\rangle- \langle F(\hat{z}_{k}), z_{k,i}-\hat{z}_{k} \rangle \\
         \leqslant&~ \|F(\hat{z}_{k})\|\|z_{k,i}-\hat{z}_{k}\|,
	\end{align*}
    which implies that
    \begin{equation*}
    	\|z_{k,i}-z^*\| \leqslant \|z_{k,i}-\hat{z}_{k}\| + \| \hat{z}_{k}-z^*\| \leqslant \sqrt{\frac{\|F(\hat{z}_{k})\|}{6H_{k,i}}} + \| \hat{z}_{k}-z^*\| \leqslant \sqrt{\frac{\|F(\hat{z}_{k})\|}{6H_{0}}}+ \| \hat{z}_{k}-z^*\|.
    \end{equation*}
	Then, by further combining the above relation  with \eqref{2:lem2:3}, it follows that the sequences  $\{z_{k,i}\}$, $\{z_k\}$ and $\{\hat{z}_k\}$ are bounded.
	By Assumption \ref{ass:smoothl}, there exists $\rho > 0$ such that  for any $(x, y),(x', y') \in \mathcal{C}:= conv\{z^*, z_{k,i}, z_k, \hat{z}_k, i,k=1, 2,\cdots\}$, we have
	\begin{align}\label{localrho}
		\| \nabla^2 f(x, y) - \nabla^2 f(x', y')\| \leqslant \rho\|(  x - x', y - y' )\|.
	\end{align}	
	By induction and Step 2(c) in Algorithm \ref{alg:1}, we then show that 
	\begin{equation}\label{2:lem3:2}
		H_k \leqslant 2\rho, \forall~ k \geqslant 0.
	\end{equation}
	 Noting that  $H_0=\frac{\|DF(z_0)-DF(\tilde{z}_0)\|}{\|z_0-\tilde{z}_0\|}$ $(z_0 \neq \tilde{z}_0)$, we can easily see that $H_0 \leqslant \rho$. Suppose \eqref{2:lem3:2} holds for $k-1$.
	 Then, at the $k$-th iteration, for the case $i_k=1$, we have $H_{k}=H_{k,1}=H_{k-1} \leqslant 2\rho$.
	 For the case $i_k >1$, by Step 2(c) and \eqref{localrho}, we have 
	\begin{align*}
	    \frac{H_{k,i_k-1}}{2}\|z_{k,i_k-1}-\hat{z}_k\|^2 \leqslant\|F(z_{k,i_k-1}) - F(\hat{z}_k)- DF(\hat{z}_k)(z_{k,i_k-1}-\hat{z}_k)\| \leqslant \frac{\rho}{2}\|z_{k,i_k-1}-\hat{z}_k\|^2,
	\end{align*}
	which means $H_k = H_{k,i_k}=2H_{k,i_k-1} \leqslant 2\rho$.
	Thus, we have $H_k \leqslant 2\rho$ and it finishes the induction.
	
	Note that $H_k \leqslant H_{k+1}$ by Step 2 in Algorithm \ref{alg:1}. By further combining with the choice that $\lambda_{k+1}H_k\|z_{k+1}-\hat{z}_{k}\| \geqslant \frac{1}{33}$, we have
	\begin{align*}
	\bigg(\frac{1}{33H_{s}}\bigg)^2\sum_{k=1}^s (\lambda_{k+1})^{-2} &\leqslant \sum_{k=1}^s (\lambda_{k+1})^{-2}\bigg(\frac{1}{33H_k}\bigg)^2 
	\leqslant \sum_{k=1}^s \|z_{k+1} - \hat{z}_{k}\|^2.
	\end{align*}
	Then by \eqref{2:lem2:2} in Lemma \ref{2:lem2}, we have
	\begin{align}\label{2:lem3:3}
	\sum_{k=1}^s (\lambda_{k+1})^{-2} \leqslant \big(66\sqrt{3}H_{s}\|z_0 - z^\star\|\big)^2.
	\end{align}
	On the other hand, by using the H\"{o}lder inequality we obtain
	\begin{align}\label{2:lem3:4}
	s=\sum_{k=1}^s 1 = \sum_{k=1}^s \big((\lambda_{k+1})^{-2}\big)^{1/3} (\lambda_{k+1})^{2/3} \leqslant \bigg(\sum_{k=1}^s (\lambda_{k+1})^{-2}\bigg)^{1/3}\bigg(\sum_{k=1}^s \lambda_{k+1}\bigg)^{2/3}.
	\end{align}
	Combining \eqref{2:lem3:3} and \eqref{2:lem3:4}, we have
	\begin{align*}
	s \leqslant \big(66\sqrt{3}H_{s}\|z_0 - z^\star\|\big)^{2/3}\left(\sum_{k=1}^s \lambda_{k+1}\right)^{2/3},
	\end{align*}
	which completes the proof. 
\end{proof}

We are now ready to establish the iteration complexity for the LF-CR algorithm  to obtain an $\epsilon$-optimal solution of \eqref{P} with respect to  the restricted primal-dual gap and gradient norm. Let $\epsilon > 0$ be any given target accuracy, we denote  the first iteration index to achieve $\textsc{gap}(\bar{x}_{N},\bar{y}_{N}; \beta) \leqslant \epsilon$  or $\|F(\hat{z}_{N})\| = 0$ by
\begin{align*}
N(\epsilon):= \min\{N |~\textsc{gap}(\bar{x}_{N},\bar{y}_{N}; \beta) \leqslant \epsilon ~\text{or}~ \|F(\hat{z}_{N})\| = 0\},
\end{align*}
where $\textsc{gap}(x,y; \beta)$ is defined as in Definition \ref{def:reGAP}, and
\begin{align*}
\bar{x}_N = \tfrac{1}{\sum_{k=1}^N \lambda_{k+1}}\bigg(\sum_{k=1}^N \lambda_{k+1} x_{k+1}\bigg), \quad \bar{y}_N = \tfrac{1}{\sum_{k=1}^N \lambda_{k+1}}\bigg(\sum_{k=1}^N \lambda_{k+1} y_{k+1}\bigg)
\end{align*}
with $\beta$ being specified later.
In addition, we denote  the first iteration index to achieve $\|F(z_{k+1})\| \leqslant \epsilon$ by 
	\begin{align*}
   \bar{N}(\epsilon):= \min\{k | \|F(z_{k+1})\| \leqslant \epsilon\}.
\end{align*}
In the following theorem, we provide an upper bound on $N(\epsilon)$ and $\bar{N}(\epsilon)$.
\begin{thm}\label{2:thm1}
	Suppose that Assumptions \ref{ass:func} and \ref{ass:smoothl} hold. Let $\{z_k\}$ and $\{\hat{z}_k\}$ be the sequences generated by Algorithm \ref{alg:1}. Set $\beta = 7\|z_0-z^*\|$, $\lambda_{k+1}=\frac{c_k}{H_k\|z_{k+1}-\hat{z}_k\|}$ with $c_k$ being a constant. If $\frac{1}{33}\leqslant c_k \leqslant \frac{1}{13}$, then
	\begin{align}\label{2:thm1:1}
	N(\epsilon) \leqslant \left(\frac{7841\sqrt{3}\rho\|z_0-z^*\|^3}{\epsilon}\right)^{2/3} + 1, ~\bar{N}(\epsilon) \leqslant \frac{156\rho\|z_0-z^*\|^2}{\epsilon}+1,
	\end{align}
	where $z^*=(x^*,y^*)$ is an optimal solution of \eqref{P} and $\rho$ is defined in Lemma \ref{2:lem3}.
\end{thm}
\begin{proof}
	Under the convex-concave setting,  we have
	\begin{align}\label{2:thm1:2}
	f(x_{k+1}, y) - f(x, y_{k+1}) =&~ f(x_{k+1},y)-f(x_{k+1},y_{k+1})  +  f(x_{k+1},y_{k+1})- f(x, y_{k+1}) \nonumber \\
	\leqslant&~ \langle z_{k+1} - z, F(z_{k+1})\rangle.
	\end{align}
	Then, we further obtain
	\begin{align}\label{2:thm1:3}	
	f(\bar{x}_{N(\epsilon)-1}, y) - f(x, \bar{y}_{N(\epsilon)-1})
	\leqslant &~ \tfrac{1}{\sum_{k=1}^{N(\epsilon)-1} \lambda_{k+1}} \sum_{k=1}^{N(\epsilon)-1} \lambda_{k+1} \big(f(x_{k+1}, y) - f(x, y_{k+1})\big)\nonumber \\
	\leqslant &~ \tfrac{1}{\sum_{k=1}^{\bar{N}(\epsilon)-1} \lambda_{k+1}} \sum_{k=1}^{N(\epsilon)-1} \lambda_{k+1} \langle z_{k+1} - z, F(z_{k+1})\rangle.
	\end{align}
    According to the definition of $N(\epsilon)$ and Lemma \ref{2:add1}, for all $k \leqslant N(\epsilon)-1$, we have $\|z_{k+1}-\hat{z}_k\| \neq 0$.
	By further combining \eqref{2:lem2:1}, \eqref{2:lem3:1} and \eqref{2:thm1:3}, we get
	\begin{align}\label{2:thm1:4}	
	f(\bar{x}_{{N(\epsilon)-1}}, y) - f(x, \bar{y}_{{N(\epsilon)-1}}) \leqslant \frac{33\sqrt{3}H_{N(\epsilon)-1}\|z_0-z\|^2\|z_0-z^*\|}{({N(\epsilon)-1)}^{3/2}}.
	\end{align}
	By \eqref{2:lem2:3} and the definition of $\bar{z}_{N(\epsilon)-1} := (\bar{x}_{N(\epsilon)-1}, \bar{y}_{N(\epsilon)-1})$, we have $\|\bar{z}_{N(\epsilon)-1}-z^*\| \leqslant \beta$.
	By \eqref{def:reGAP:1} and \eqref{2:thm1:4}, we obtain
	\begin{align*}
	&~~\quad\textsc{gap}(\bar{x}_{N(\epsilon)-1},\bar{y}_{N(\epsilon)-1}; \beta) \nonumber \\
	&~ = \max_{x \in \mathbb{B}_\beta(x^\star),y \in \mathbb{B}_\beta(y^\star)} (f(\bar{x}_{N(\epsilon)-1}, y)-f(x, \bar{y}_{N(\epsilon)-1})) \nonumber \\
	&~ \leqslant \max_{x \in \mathbb{B}_\beta(x^\star),y \in \mathbb{B}_\beta(y^\star)} 	\frac{33\sqrt{3}H_{N(\epsilon)-1}\|z_0-z\|^2\|z_0-z^*\|}{(N(\epsilon)-1)^{3/2}} \nonumber\\
	&~ \leqslant \max_{x \in \mathbb{B}_\beta(x^\star),y \in \mathbb{B}_\beta(y^\star)}\frac{33\sqrt{3}H_{N(\epsilon)-1}(\|z_0-z^*\| + \sqrt{2} \beta)^2\|z_0-z^*\|}{(N(\epsilon)-1)^{3/2}} \nonumber \\
	&~= \frac{33\sqrt{3}(7\sqrt{2}+1)^2H_{N(\epsilon)-1}\|z_0-z^*\|^3}{(N(\epsilon)-1)^{3/2}}.
	\end{align*}
	Note that $H_{k} \leqslant 2\rho$ for any $k \geqslant 1$. Then by the definition of $N(\epsilon)$, we immediately have
	\begin{align*}
	N(\epsilon) \leqslant \left(\frac{33\sqrt{3}(7\sqrt{2}+1)^2H_{N(\epsilon)-1}\|z_0-z^*\|^3}{\epsilon}\right)^{2/3} + 1  \leqslant \left(\frac{7841\sqrt{3}\rho\|z_0-z^*\|^3}{\epsilon}\right)^{2/3} + 1,
	\end{align*}
which shows the first inequality of \eqref{2:thm1:1}.
    
    Next, we establish the iteration complexity of Algorithm \ref{alg:1} with respect to  the gradient norm. By the definition of $\bar{N}(\epsilon)$ and Lemma \ref{2:add1}, for all $k \leqslant \bar{N}(\epsilon)-1$, we have $\|z_{k+1}-\hat{z}_k\| \neq 0$.
    According to \eqref{2:add1:1} and \eqref{2:add1:2}, it follows that
    \begin{align*}
    	\|F(z_{k+1})\| \leqslant& \|F(z_{k+1}) - F(\hat{z}_{k})- DF(\hat{z}_{k})(z_{k+1}-\hat{z}_{k})\| + \|F(\hat{z}_{k}) + DF(\hat{z}_{k})(z_{k+1}-\hat{z}_{k})\| \\
    	\leqslant& \frac{H_k}{2}\|z_{k+1}-\hat{z}_{k}\|^2 + 6H_k\|z_{k+1}-\hat{z}_{k}\|^2=\frac{13H_k}{2}\|z_{k+1}-\hat{z}_{k}\|^2.
    \end{align*}
	Combining the above relation with \eqref{2:lem2:2}, we have
	\begin{align*}
			\sum_{k=1}^{\bar{N}(\epsilon)-1} \frac{2}{13H_k}\|F(z_{k+1})\|\leqslant \sum_{k=1}^{\bar{N}(\epsilon)-1} \|z_{k+1} - \hat{z}_{k}\|^2 \leqslant 12 \|z_0-z^*\|^2.
	\end{align*}
    Note that $H_{k} \leqslant 2\rho$ for all $k \geqslant 1$. Then by the definition of $\bar{N}(\epsilon)$, we immediately have
    \begin{align*}
    	\bar{N}(\epsilon) \leqslant \frac{156\rho\|z_0-z^*\|^2}{\epsilon}+1,
    \end{align*}
    which completes the proof.
\end{proof}
\begin{rem}
Note that for any $k\geqslant 1$, we have $H_0 \leqslant H_k \leqslant H_{k+1} \leqslant  2\rho$. At the $k$-th iteration, the number of line searches for $H_k$ is upper bounded by $i_k=1+\log_2\frac{H_k}{H_{k-1}}$. Then, we have $\sum_{k=1}^{N(\epsilon)}i_k= N(\epsilon)+\log_2\frac{H_{N(\epsilon)+1}}{H_0} \leqslant N(\epsilon)+\log_2 \frac{2\rho}{H_0}.$ Theorem \ref{2:thm1} implies that the total iteration complexity of the LF-CR algorithm to obtain an $\epsilon$-optimal solution to \eqref{P} with respect to the restricted primal-dual gap and gradient norm is upper bounded by $\mathcal{O}(\rho^{2/3}\|z_0-z^*\|^2\epsilon^{-2/3})$ and $\mathcal{O}(\rho\|z_0-z^*\|^2\epsilon^{-1})$, respectively.
\end{rem}
Compared to the Newton-MinMax algorithm in \citep{Lin2022ExplicitSM}, the LF-CR algorithm achieves the same  complexity for finding an $\epsilon$-optimal solution to \eqref{P} without requiring prior knowledge of the Lipschitz constant, assuming that the Hessian is locally Lipschitz continuous. 
Moreover, the LF-CR algorithm  has better iteration complexity with respect to the Lipschitz constant $\rho$ compared to the ASOM method in \citep{Jiang2024AdaptiveAO}, which has an iteration complexity of $\mathcal{O}\left((\ell^{4/3}\rho^{2/3}\|z_0-z^*\|^{2/3} + \rho^2 \|z_0-z^*\|^2 )\epsilon^{-2/3}\right)$ (resp. $\mathcal{O}\left((\ell\rho\|z_0-z^*\| + \rho^2 \|z_0-z^*\|^2 )\epsilon^{-1}\right)$)  in terms of the $\epsilon$-optimal solution of the restricted primal-dual gap (resp. gradient norm).

\section{A Fully Parameter-Free Cubic Regularization (FF-CR) Algorithm }
In this section, we further propose a fully parameter-free cubic regularization (FF-CR) algorithm that does not require any prior knowledge about the problem parameters. We show that the proposed algorithm achieves the complexity bound  $\mathcal{O}(\epsilon^{-2/3})$ with respect to the gradient norm in the absence of such information.

Based on the LF-CR algorithm, the proposed FF-CR algorithm further employs an accumulative regularization method to construct the subproblem and uses a ``guess-and-check" strategy to estimate $\|z_0-z^*\|$ step by step. Such accumulative regularization method and ``guess-and-check" strategy have also been used to design a first-order parameter-free gradient minimization algorithm for convex optimization problem in \citep{Lan2023OptimalAP}. 

Each iteration of the proposed FF-CR algorithm consists of two loops, the outer loop and the inner loop. 
In the outer loop, the estimated bound $D_t$ of $\|z_0-z^*\|$ iteratively increases. In the inner loop, for a given estimate $D_t$, it solves the accumulative regularization subproblem approximately without prior knowledge of the Lipschitz constants. More specifically, it performs the following two important algorithmic components at each inner loop: 

\begin{itemize}
	\item Compute an approximate solution $z_k^t$ of the following regularized minimax subproblem:
	\begin{align}\label{minmax:1}
		\min_{x\in\mathbb{R}^{m}} \max_{y\in\mathbb{R}^{n}} f_k^t(x,y):=f(x,y) + \dfrac{\sigma_k^t}{2}\|x - \bar{x}_k^t\|^2 - \dfrac{\sigma_k^t}{2}\|y - \bar{y}_k^t\|^2,
	\end{align}
	where $\bar{z}_k^t=  (1-\gamma_k^t) \bar{z}_{k-1}^t + \gamma_k^t z_{k-1}^t$ that is a convex combination of previous approximate solutions $\{z_0^t, \cdots, z_{k-1}^t\}$ and $\sigma_k^t$ is the regularization parameter, which will be specified later. The problem \eqref{minmax:1} is solved by running the LF-CR algorithm within $N_k^t$ times iteration.
	\item At $z_k^t$, search a pair of $(M_{k,i}^t,z_{k,i}^t)$ by backtracking such that: 
	\begin{align}\label{minmax:2}
		\|F_k^t(z_{k,i}^t) - F_k^t(z_k^t)- DF_k^t(z_k^t)(z_{k,i}^t-z_k^t)\| \leqslant \frac{M_{k,i}^t}{2}\|z_{k,i}^t-z_k^t\|^2,
	\end{align}
	where $z_{k,i}^t$ is a solution of the following nonlinear equation:
	\begin{align}\label{minmax:3}
		F_k^t(z_k^t) + DF_k^t(z_k^t)(z-z_k^t) + 6M_{k,i}^t\| z-z_k^t \|( z-z_k^t)=0.
	\end{align} 
\end{itemize} 

\begin{algorithm}[H]
	\caption{ Fully parameter-free cubic regularization (FF-CR) algorithm}
	\label{alg2}
	\begin{algorithmic}
		\STATE{\textbf{Step 1}: Input $z_0, \tilde{z}_0$; Set $t=0$.}
		\STATE{\textbf{Step 2}: Initialize $M_0$, $D_0$.}
			\STATE{\quad\textbf{(a)}: Set $i=1$, $M_{0,i}=\frac{\|DF(z_0)-DF(\tilde{z}_0)\|}{\|z_0-\tilde{z}_0\|} (z_0 \neq \tilde{z}_0)$.}
			\STATE{\quad\textbf{(b)}: Update $z_{0,i}$ such that it is a solution of the following nonlinear equation:	
				\begin{equation*}
					F(z_0)+ DF(z_0)(z-z_0)+ 6M_{0,i}\|z-z_0\|(z-z_0)=0.
			\end{equation*}}
			\STATE{\quad\textbf{(c)}: If 
				\begin{equation*}
					\|F(z_{0,i}) - F(z_0)- DF(z_0)(z_{0,i}-z_0)\| \leqslant \frac{M_{0,i}}{2}\|z_{0,i}-z_0\|^2,
				\end{equation*}
				\qquad\quad~$M_{0}=M_{0,i}$,  $D_0=\|z_{0,i}-z_0\|$; otherwise, set $M_{0,i+1}=2M_{0,i}$, $i=i+1, $ go to \\~\quad\quad\quad Step 2(b).}
		\STATE{\textbf{Step 3}: For  a given $D_t$, solve the accumulative regularized subproblem:}
		\STATE{\quad\textbf{(a)}: Set $\bar{z}_0^t=z_0^t=z_0$, $\sigma_{0}^t=0$, $H_0^t=M_0^t=M_0$, $k=1$. }
		\STATE{\quad\textbf{(b)}: Input $\sigma_k^t$. Update $\gamma_{k}^t,\bar{x}_k^t,\bar{y}_k^t$:
			\begin{align}\label{alg2:1}
				\gamma_k^t = 1 - \frac{\sigma_{k-1}^t}{\sigma_k^t}, \bar{x}_k^t =  (1-\gamma_k^t) \bar{x}_{k-1}^t + \gamma_k^t x_{k-1}^t,\bar{y}_k^t =  (1-\gamma_k^t) \bar{y}_{k-1}^t + \gamma_k^t y_{k-1}^t.
		\end{align}}
		\STATE{\quad\textbf{(c)}:  Running the {\bfseries LF-CR} algorithm within $N_k^t$ iterations to approximately solve
			\begin{align}
				\min_{x\in\mathbb{R}^{m}} \max_{y\in\mathbb{R}^{n}} f_k^t(x,y):=f(x,y) + \dfrac{\sigma_k^t}{2}\|x - \bar{x}_k^t\|^2 - \dfrac{\sigma_k^t}{2}\|y - \bar{y}_k^t\|^2,
			\end{align}
			\qquad\quad with $z_{k-1}^t$ as the initial point and $H_{k-1}^t$ as the initial guess of Hessian Lipschitz \\ \qquad\quad constant.			Denote the output of the LF-CR algorithm as $(z_k^t,H_k^t)$. }
		\STATE{\quad\textbf{(d)}: Find a pair $(M_{k,i}^t,z_{k,i}^t)$ to satisfy \eqref{minmax:2}:}
		\STATE{\quad\quad~\textbf{(i)}: Set $i = 1$. If $z_k^t \neq z_0^t$, set $M_{k,i}^t =$  $\max  \left\{M_{k-1}^t,\frac{\|D F(z_k^t)-DF(z_0^t)\|}{\|z_k^t-z_0^t\|}\right\}$
		 ; otherwise, \\\quad\quad\quad\quad set $M_{k,i}^t = M_{k-1}^t$.}
		\STATE{\quad\quad~\textbf{(ii)}: Update $z_{k,i}^t$ such that it is a solution of
			the following nonlinear equation:			
			\begin{align}\label{alg2:2}
				F_k^t(z_k^t) + DF_k^t(z_k^t)(z-z_k^t) + 6M_{k,i}^t\left\| z-z_k^t \right\|( z-z_k^t)=0.
		\end{align}} 
		\STATE{\quad\quad~\textbf{(iii)}: If
			\begin{align}\label{alg2:3}
				\|F_k^t(z_{k,i}^t) - F_k^t(z_k^t)- DF_k^t(z_k^t)(z_{k,i}^t-z_k^t)\| \leqslant \frac{M_{k,i}^t}{2}\|z_{k,i}^t-z_k^t\|^2,
			\end{align}
			\quad\quad\quad\quad  $M_k^t=M_{k,i}^t, \tilde{z}=z_{k,i}^t$; otherwise, set $M_{k,i+1}^t=2M_{k,i}^t, i=i+1, $ go to Step 3(d)(ii).}
		\STATE{\quad\textbf{(e)}: If $k \geqslant K_t$, set $(z_{t+1},M_{t+1})= (z_k^t,M_k^t)$; otherwise, set $k=k+1$, go to Step 3(b).}		
		\STATE{\textbf{Step 4}: If $\|F(z_{t+1})\| \leqslant \epsilon $, stop; otherwise, set $D_{t+1}=4D_{t}, t=t+1$, go to Step 3(a).}	
	\end{algorithmic}
\end{algorithm}

Note that \eqref{minmax:3} can be solved in a similar way to \eqref{nesubproblem}. The detailed algorithm is formally stated in Algorithm \ref{alg2}.

In the following subsection, we will establish the iteration complexity of the FF-CR algorithm to obtain an $\epsilon$-optimal solution of \eqref{P} with respect to  the gradient norm.

\subsection{Convergence analysis}
Before proving the iteration complexity of the FF-CR algorithm, we make the following assumption about $f(x,y).$
\begin{ass}\label{ass:smoothg}
	The function $f(x, y)$ is globally Hessian Lipschitz continuous.
\end{ass}
We then outline the basic idea of iteration complexity analysis. We first prove that once $D_t$ in Step 3 of Algorithm \ref{alg2} satisfies $D_t\geqslant\|z_0-z^*\|$, we can prove $\|\nabla f(x_{t+1}, y_{t+1})\| \leqslant \epsilon$ and the algorithm will stop. Otherwise, $D_t\leqslant 4\sqrt{\frac{12}{11}}\|z_0-z^*\|$. Then, we estimate the additional number of iterations for such bounded $D_t$,  and finally prove the iteration complexity of Algorithm \ref{alg2}. 

First, we prove the following lemma which shows some bounds on  $\sigma_k^t\|\bar{z}_k^t - (z_k^t)^*\|$ at the $t$-th iteration, where $(z_k^t)^*$ is the optimal solution of \eqref{minmax:1}. 

\begin{lem}\label{3:lem1}
	Let $\{z_k^t\}$ and $\{\bar{z}_k^t\}$ be the sequences generated by Algorithm \ref{alg2}. At the $t$-th iteration, for all $k \geqslant 1$, we have 
	\begin{align}
	&\|z_{k-1}^t - (z_k^t)^*\| \leqslant \|z_{k-1}^t - (z_{k-1}^t)^*\|, \label{3:lem1:1}\\
	&\sigma_k^t\|\bar{z}_k^t - (z_k^t)^*\|  \leqslant \sum_{i=1}^{k}(\sigma_{i-1}^t + \sigma_{i}^t)\|z_{i-1}^t - (z_{i-1}^t)^*\|,\label{3:lem1:2}
	\end{align} 
	where $(z_k^t)^* = \arg\min_{x}\max_{y}f_k^t(x,y):=f(x,y) + \dfrac{\sigma_k^t}{2}\|x - \bar{x}_k^t\|^2 - \dfrac{\sigma_k^t}{2}\|y - \bar{y}_k^t\|^2.$
\end{lem}

\begin{proof}
	By \eqref{alg2:1} in Algorithm \ref{alg2}, we get
	\begin{align}\label{3:lem1:3}
	(z_k^t)^*
	=&~\arg\min_{x}\max_{y} f(x,y) + \dfrac{\sigma_k^t}{2}\|x - \bar{x}_k^t\|^2 - \dfrac{\sigma_k^t}{2}\|y - \bar{y}_k^t\|^2 \nonumber \\
	=&~\arg\min_{x}\max_{y}  f(x,y) + \dfrac{\sigma_{k-1}^t}{2}\|x - \bar{x}_{k-1}^t\|^2 + \dfrac{\sigma_k^t - \sigma_{k-1}^t}{2}\|x - x_{k-1}^t\|^2  \nonumber\\
	&~ - \dfrac{\sigma_{k-1}^t}{2}\|y - \bar{y}_{k-1}^t\|^2- \dfrac{\sigma_k^t - \sigma_{k-1}^t}{2}\|y - y_{k-1}^t\|^2 \nonumber \\
	=&~\arg\min_{x}\max_{y} f_{k-1}^t(x,y) + \frac{\gamma_{k}^t\sigma_k^t}{2}\|x - x_{k-1}^t\|^2 - \frac{\gamma_{k}^t\sigma_k^t}{2}\|y - y_{k-1}^t\|^2. 
	\end{align}
	By the optimality condition of $(z_k^t)^*$ in \eqref{3:lem1:3}, we then have
	\begin{align}\label{3:lem1:4}
	F_{k-1}^t\big((z_k^t)^*\big) + \gamma_{k}^t\sigma_k^t \big((z_k^t)^*-z_{k-1}^t\big) = 0.
	\end{align}
	Under the convex-concave setting, $f_{k-1}^t(x,y) + \frac{\gamma_{k}^t\sigma_k^t}{2}\|x - x_{k-1}^t\|^2 - \frac{\gamma_{k}^t\sigma_k^t}{2}\|y - y_{k-1}^t\|^2$ is a $\gamma_{k}^t\sigma_k^t$ strongly convex-strongly concave function. Then, by further combining with \eqref{3:lem1:4}, we have
	\begin{align}\label{3:lem1:5}
	&~\gamma_{k}^t\sigma_k^t \|(z_{k-1}^t)^* - (z_k^t)^* \|^2  \nonumber \\
	\leqslant &~ \left\langle F_{k-1}^t\big((z_{k-1}^t)^*\big) + \gamma_{k}^t\sigma_k^t \big((z_{k-1}^t)^*-z_{k-1}^t\big)-F_{k-1}^t\big((z_k^t)^*\big) - \gamma_{k}^t\sigma_k^t \big((z_k^t)^*-z_{k-1}^t\big), \right. \nonumber \\
	&~ \left. (z_{k-1}^t)^* - (z_k^t)^*  \right\rangle \nonumber \\
	= &~\left\langle F_{k-1}^t\big((z_{k-1}^t)^*\big) + \gamma_{k}^t\sigma_k^t \big((z_{k-1}^t)^*-z_{k-1}^t\big), (z_{k-1}^t)^* - (z_k^t)^*  \right\rangle.
	\end{align}
	By the definition of $(z_{k-1}^t)^*$, we get $F_{k-1}^t\big((z_{k-1}^t)^*\big)=0$. 
	Then, we further have
	\begin{align*}
	&~\gamma_{k}^t\sigma_k^t \|(z_{k-1}^t)^* - (z_k^t)^* \|^2 \\ \leqslant &~ \gamma_{k}^t\sigma_k^t \left\langle   (z_{k-1}^t)^*-z_{k-1}^t, (z_{k-1}^t)^* - (z_k^t)^* \right\rangle \\
	= &~ \frac{\gamma_{k}^t\sigma_k^t}{2}\left(\|(z_{k-1}^t)^*-z_{k-1}^t\|^2 + \|(z_{k-1}^t)^* - (z_k^t)^*\|^2 - \|z_{k-1}^t-(z_k^t)^*\|^2\right),
	\end{align*}
	which implies that
	\begin{align}\label{3:lem1:6}
	\|(z_{k-1}^t)^* - (z_k^t)^* \|^2 + \|z_{k-1}^t-(z_{k}^t)^*\|^2 \leqslant \|z_{k-1}^t-(z_{k-1}^t)^*\|^2.
	\end{align}
	Then, \eqref{3:lem1:1} holds directly from \eqref{3:lem1:6}.
	Denoting $\alpha_k^t:=\sigma_k^t - \sigma_{k-1}^t$ and noting that $\gamma_k^t =  1 - \frac{\sigma_{k-1}^t}{\sigma_k^t}=\frac{\alpha_k^t}{\sigma_k^t}$, we can rewrite the definition of $\bar{z}_k^t$ in \eqref{alg2:1} of Algorithm \ref{alg2} to
	\begin{align}\label{3:lem1:7}
	\sigma_k^t \bar{z}_k^t = \sigma_{k-1}^t\bar{z}_{k-1}^t + \alpha_k^t z_{k-1}^t.
	\end{align}
	By $\sigma_0^t=0$, $\alpha_k^t:=\sigma_k^t - \sigma_{k-1}^t$ and \eqref{3:lem1:7}, we have
	\begin{align}\label{3:lem1:8}
	\sigma_k^t = \sum_{i=1}^{k}\alpha_i^t, ~\sigma_k^t \bar{z}_k^t = \sum_{i=1}^{k}\alpha_i^t z_{i-1}^t.
	\end{align}
	Then, by \eqref{3:lem1:8}, we obtain
	\begin{align}\label{3:lem1:9}
	&~\sigma_k^t\big(\bar{z}_k^t - (z_k^t)^*\big) \nonumber\\
	=~ & \sum_{i=1}^{k}\alpha_i^t\big(z_{i-1}^t- (z_k^t)^*\big) \nonumber\\
	=~ & \alpha_k^t\big(z_{k-1}^t - (z_k^t)^*\big) + \sum_{i=1}^{k-1}\alpha_i^t\big(z_{i-1}^t - (z_{k-1}^t)^*\big)  + \left(\sum_{i=1}^{k-1}\alpha_i^t\right)\big((z_{k-1}^t)^* - (z_k^t)^*\big)
	\nonumber\\
	=~ & \alpha_k^t\big(z_{k-1}^t - (z_k^t)^*\big) + \sigma_{k-1}^t\big(\bar{z}_{k-1}^t - (z_{k-1}^t)^*\big)
	+ \sigma_{k-1}^t\big((z_{k-1}^t)^* - z_{k-1}^t\big) \nonumber\\
	&+ \sigma_{k-1}^t\big(z_{k-1}^t - (z_k^t)^*\big)
	\nonumber\\
	=~ & \sigma_k^t\big(z_{k-1}^t - (z_k^t)^*\big) + \sigma_{k-1}^t\big(\bar{z}_{k-1}^t - (z_{k-1}^t)^*\big) + \sigma_{k-1}^t\big((z_{k-1}^t)^* - z_{k-1}^t\big).
	\end{align} 
	By \eqref{3:lem1:9} and \eqref{3:lem1:1}, we get
	\begin{align*}
	&~\sigma_k^t\|\bar{z}_k^t - (z_k^t)^*\|  \\
	\leqslant&~ \sigma_k^t\|z_{k-1}^t - (z_k^t)^*\| + \sigma_{k-1}^t \|\bar{z}_{k-1}^t - (z_{k-1}^t)^*\| + \sigma_{k-1}^t \|(z_{k-1}^t)^* - z_{k-1}^t\| \\
	\leqslant&~ (\sigma_{k-1}^t + \sigma_{k}^t )\|z_{k-1}^t - (z_{k-1}^t)^*\| + \sigma_{k-1}^t \|\bar{z}_{k-1}^t - (z_{k-1}^t)^*\|.
	\end{align*}
	Then, \eqref{3:lem1:2} is proved by combining with $\sigma_0^t=0$.
\end{proof}

The following lemma gives an upper bound of the gradient norm at the point $z_k^t$ by using the distances between $z_i^t$ and $(z_i^t)^*$, $i = 0,\cdots, k$.

\begin{lem}\label{3:lem2}
	Suppose that Assumptions \ref{ass:func} and \ref{ass:smoothg} hold. Let $\{z_k^t\}$ and $\{\bar{z}_k^t\}$ be the sequences generated by Algorithm \ref{alg2}. At the $t$-th iteration, for all $k \geqslant 1$, we have
	\begin{align}\label{3:lem2:1}
	\|F(z_k^t)\| \leqslant 
	&~ \bigg(\sqrt{\frac{12}{11}} + \frac{72}{11}\bigg) M_k^t \|z_k^t-(z_k^t)^*\|^2 + 2 \sqrt{\frac{12}{11}} M_k^t \bigg(\sum_{i=1}^{k-1}\|z_i^t-(z_i^t)^*\|\bigg)\|z_k^t-(z_k^t)^*\| \nonumber \\
	&~ + \sqrt{\frac{12}{11}} M_k^t \|z_0^t - (z_0^t)^*\|\|z_k^t-(z_k^t)^*\| + \bigg(\sqrt{\frac{12}{11}} +1\bigg) \sigma_{k}^t\|z_k^t-(z_k^t)^*\|  \nonumber\\
	&~ + \sqrt{\frac{12}{11}}\|DF(z_0^t)\|\|z_k^t-(z_k^t)^*\| + \sigma_1^t\|z_0^t -(z_0^t)^*\| \nonumber\\
	&~ + \sum_{i=2}^{k}(\sigma_{i-1}^t + \sigma_{i}^t)\|z_{i-1}^t - (z_{i-1}^t)^*\|.
	\end{align}
\end{lem}

\begin{proof}
	By Step 3(d) in Algorithm \ref{alg2}, we have 
	\begin{align}
	&F_k^t(z_k^t) + DF_k^t(z_k^t)(\tilde{z}-z_k^t) + 6M_{k}^t \|\tilde{z}-z_k^t\|(\tilde{z}-z_k^t) = 0,\label{3:lem2:2}\\
	&\|F_k^t(\tilde{z}) - F_k^t(z_k^t)- DF_k^t(z_k^t)(\tilde{z}-z_k^t)\| \leqslant \frac{M_k^t}{2}\|\tilde{z}-z_k^t\|^2. \label{3:lem2:3}
	\end{align}
	Under the convex-concave setting, by the optimality condition of $(z_k^t)^*$, we further have
	\begin{align*}
	0 \leqslant&~ \langle F_k^t(\tilde{z}), \tilde{z}-(z_k^t)^* \rangle \nonumber \\
	= &~ \langle F_k^t(\tilde{z}) - F_k^t(z_k^t) - DF_k(z_k)(\tilde{z}-z_k^t), \tilde{z}-(z_k^t)^*\rangle \nonumber \\
	&~ + \langle F_k^t(z_k^t) + DF_k^t(z_k^t)(\tilde{z}-z_k^t), \tilde{z}-(z_k^t)^*\rangle \nonumber \\
	\leqslant&~ \frac{M_k^t}{2}\|\tilde{z}-z_k^t\|^2\|\tilde{z}-(z_k^t)^*\| - 6M_k^t\|\tilde{z}-z_k^t\|\langle \tilde{z}-z_k^t, \tilde{z}-(z_k^t)^*\rangle,
	\end{align*}
	which implies that 
	\begin{align}\label{3:lem2:4}
	0 \leqslant \|\tilde{z}-z_k^t\|\|\tilde{z}-(z_k^t)^*\| - 12\langle \tilde{z}-z_k^t, \tilde{z}-(z_k^t)^*\rangle.
	\end{align}
	By \eqref{3:lem2:4} and the Cauchy-Schwarz inequality, we obtain
	\begin{align*}
	0 \leqslant&~ \frac{1}{2}\|\tilde{z}-z_k^t\|^2 + \frac{1}{2}\|\tilde{z}-(z_k^t)^*\|^2 - 6\big(\|\tilde{z}-z_k^t\|^2 + \|\tilde{z}-(z_k^t)^*\|^2 - \|z_k^t - (z_k^t)^*\|^2\big) \nonumber \\
	=&~ -\frac{11}{2}\|\tilde{z}-z_k^t\|^2 - \frac{11}{2}\|\tilde{z}-(z_k^t)^*\|^2 + 6\|z_k^t - (z_k^t)^*\|^2,
	\end{align*}
	which implies that
	\begin{align}\label{3:lem2:5}
	\|\tilde{z}-z_k^t\|^2 \leqslant \frac{12}{11} \|z_k^t - (z_k^t)^*\|^2.
	\end{align}
	Then, by \eqref{3:lem2:5} and \eqref{3:lem2:2}, we have
	\begin{align}\label{3:lem2:6}
	\|F_k^t(z_k^t)\| =&~ \|DF_k^t(z_k^t)(\tilde{z}-z_k^t) +6M_k^t\|\tilde{z}-z_k^t\|(\tilde{z}-z_k^t)\| \nonumber\\
	\leqslant &~ \|DF_k^t(z_k^t)\|\|\tilde{z}-z_k^t\| + 6M_k^t\|\tilde{z}-z_k^t\|^2 \nonumber\\
	\leqslant &~ \sqrt{\frac{12}{11}}\|DF_k^t(z_k^t)\|\|z_k^t-(z_k^t)^*\|+\frac{72M_k^t}{11}\|z_k^t-(z_k^t)^*\|^2.
	\end{align}
	By combining \eqref{3:lem2:6} with $DF_k^t(z_k^t)=DF(z_k^t)+\sigma_k^t I_{m+n}$ and $\frac{\|D F(z_k^t)-DF(z_0^t)\|}{\|z_k^t-z_0^t\|} \leqslant M_{k,i}^t \leqslant M_k^t$ in Step 3(d) of Algorithm \ref{alg2}, we have
	\begin{align}\label{3:lem2:7}
	\|F_k^t(z_k^t)\| 
	\leqslant&~ \sqrt{\frac{12}{11}}\|DF(z_k^t) + \sigma_k^t I_{m+n} - DF(z_0^t) +DF(z_0^t) \|\|z_k^t-(z_k^t)^*\| \nonumber\\
	&~+ \frac{72M_k^t}{11}\|z_k^t-(z_k^t)^*\|^2 \nonumber \\
	\leqslant&~ \sqrt{\frac{12}{11}}M_k^t\|z_k^t-z_0^t\|\|z_k^t-(z_k^t)^*\| + \sqrt{\frac{12}{11}}\big(\sigma_k^t + \|DF(z_0^t)\|\big)\|z_k^t-(z_k^t)^*\| \nonumber \\
	&~+ \frac{72M_k^t}{11}\|z_k^t-(z_k^t)^*\|^2.
	\end{align}
	Next, we are ready to give an upper bound on $\|F(z_k^t)\|$. By combining \eqref{3:lem2:7} with
	\begin{align*}
	\|F(z_k^t)\| = \|F_k^t(z_k^t) -\sigma_{k}^t(z_k^t - \bar{z}_k^t)\| 
	\leqslant \|F_k^t(z_k^t)\| + \sigma_{k}^t\|z_k^t - (z_k^t)^*\| + \sigma_{k}^t\|\bar{z}_k^t - (z_k^t)^*\|,
	\end{align*}
	we have
	\begin{align}\label{3:lem2:8}
	\|F(z_k^t)\| \leqslant &~\sqrt{\frac{12}{11}}M_k^t\|z_k^t-z_0^t\|\|z_k^t-(z_k^t)^*\| + \sqrt{\frac{12}{11}}(\sigma_k^t + \|DF(z_0^t)\|)\|z_k^t-(z_k^t)^*\| \nonumber \\
	&~+ \frac{72M_k^t}{11}\|z_k^t-(z_k^t)^*\|^2 + \sigma_{k}^t\|z_k^t - (z_k^t)^*\| + \sigma_{k}^t\|\bar{z}_k^t - (z_k^t)^*\|.
	\end{align}
	By \eqref{3:lem1:1}, we can easily prove that 
	\begin{align} \label{3:lem2:9}
	&~\|z_k^t - z_0^t\| \nonumber\\
	=&~\|z_k^t-(z_k^t)^* + (z_k^t)^*-z_{k-1}^t+z_{k-1}^t-(z_{k-1}^t)^* + \cdots + z_1^t-(z_1^t)^*+(z_1^t)^*-z_0^t\| \nonumber\\ 
	\leqslant&~ \|z_k^t-(z_k^t)^*\| + \sum_{i=1}^{k-1}2\|z_i^t-(z_i^t)^*\| + \|z_0^t -(z_0^t)^*\|.
	\end{align}
	By plugging \eqref{3:lem2:9} and \eqref{3:lem1:2} into \eqref{3:lem2:8}, the proof is then completed. 
\end{proof}

By applying the convergence properties of the LF-CR algorithm, the following lemma shows that $\|F(z_{t+1})\| \leqslant \epsilon$ when $D_t \geqslant \|z_0-z^*\|$.
\begin{lem}\label{3:lem3}
	Suppose that Assumptions \ref{ass:func} and \ref{ass:smoothg} hold. Let $\{z_k^t\}$ and $\{\bar{z}_k^t\}$ be the sequences generated by Algorithm \ref{alg2}. Set 
	\begin{align*}
	&\sigma_k^t = \frac{\epsilon}{41D_t}4^k, 
	N_k^t =  \left\lceil \left( \frac{33\sqrt{3}8^{3-k}H_k^tD_t}{\sigma_{k}^t}\right)^{2/3}  \right\rceil,  \\
	&K_t= \left\lceil \max \left\{\log_{64}\frac{32M_k^tD_t^2}{\epsilon}, \log_8\frac{8M_k^tD_t^2}{\epsilon},\log_8 \frac{4\sqrt{12/11}\|DF(z_0)\|D_t}{\epsilon} \right\}\right\rceil.
	\end{align*}    	
	At the $t$-th iteration, if $D_t \geqslant \|z_0-z^*\|$, then we have $\|F(z_{t+1})\| \leqslant \epsilon$,  and $\|z_{t+1}-z^*\| \leqslant  \frac{135}{56}\|z_0-z^*\|$,
	where $z^*=(x^*,y^*)$ is an optimal solution of \eqref{P}.
\end{lem}
\begin{proof}
	By Step 3(c), $(x_k^t,y_k^t,H_k^t)$ is an output of the LF-CR algorithm within $N_k^t$ iterations and $f_k^t$ is $\rho$-Hessian Lipschitz continuous. 
	We consider two output case of Algorithm \ref{alg:1}, i.e., $\|F_k^t(z_k^t)\| \neq 0$ or $\|F_k^t(z_k^t)\| = 0$. For the case $\|F_k^t(z_k^t)\| \neq 0$, substituting $f=f_k^t, z=(z_k^t)^*$, $z^*=(z_k^t)^*, z_0=z_{k-1}^t, N(\epsilon)-1=N_k^t, H_{N(\epsilon)-1}=H_k^t$, \text{and}  $\bar{z}_{N(\epsilon)-1}=z_k^t$ into \eqref{2:thm1:4}, we have
	\begin{align}\label{3:lem3:1}
	f_k^t\big(x_k^t, (y_k^t)^*\big) - f_k^t\big((x_k^t)^*, y_k^t\big) \leqslant \frac{33\sqrt{3}H_k^t\|z_{k-1}^t-(z_k^t)^*\|^3}{(N_k^t)^{3/2}}.
	\end{align}
	Under the convex-concave setting, $f_k^t$ is $\sigma_k^t$ strongly convex-strongly concave. Then, we further have
	\begin{align}\label{3:lem3:2}
	\|z_k^t-(z_k^t)^*\|^2 &\leqslant \frac{33\sqrt{3}H_k^t\|z_{k-1}^t-(z_k^t)^*\|^3}{\sigma_k^t (N_k^t)^{3/2}} \leqslant \frac{33\sqrt{3}H_k^t\|z_{k-1}^t-(z_{k-1}^t)^*\|^3}{\sigma_k^t (N_k^t)^{3/2}},
	\end{align}
	where the second inequality is by \eqref{3:lem1:1}. For the case $\|F_k^t(z_k^t)\| = 0$, we know that $z_k^t$ = $(z_k^t)^*$, so \eqref{3:lem3:2} obviously holds.
	
	Next, by \eqref{3:lem3:2} and induction, we show that for any $k \geqslant 1$, we have
	\begin{align}\label{3:lem3:3}
	\|z_k^t-(z_k^t)^*\| \leqslant \frac{1}{8}\|z_{k-1}^t-(z_{k-1}^t)^*\|.
	\end{align}
	If $D_t \geqslant \|z_0-z^*\|$, combining \eqref{3:lem3:2} with the choice of $N_1^t$, $z_0^t=z_0$ and $ (z_0^t)^*=z^*$, we can easily prove that $\|z_1^t-(z_1^t)^*\| \leqslant \frac{1}{8}\|z_{0}-z^*\|$.
	Next, suppose that \eqref{3:lem3:3} holds for $k$. Then, we have
	\begin{align}\label{3:lem3:4}
	\|z_k^t-(z_k^t)^*\| \leqslant \frac{1}{8^k}\|z_0-z^*\|.
	\end{align}
	By  $D_t \geqslant \|z_0-z^*\|$, combining \eqref{3:lem3:2} with the choice of $N_{k+1}^t$ and \eqref{3:lem3:4}, we have 
	\begin{align*}
	\|z_{k+1}^t-(z_{k+1}^t)^*\|^2 \leqslant \frac{\|z_k^t-(z_k^t)^*\|^3}{8^{2-k}D_t} \leqslant \frac{\|z_k^t-(z_k^t)^*\|^2}{8^{2}},
	\end{align*}
	which implies that $\|z_{k+1}^t-(z_{k+1}^t)^*\| \leqslant \frac{1}{8}\|z_{k}^t-(z_{k}^t)^*\|$. 
	Next, we further bound the norm $\|F(z_k^t)\|$ in \eqref{3:lem2:1} by using \eqref{3:lem3:3}.
	By \eqref{3:lem3:3} and the choice of $\sigma_k^t$, we have 
	\begin{align}
	&\sum_{i=1}^{k-1}\|z_i^t-(z_i^t)^*\| \leqslant \sum_{i=1}^{k-1} \frac{1}{8^i}\|z_0-z^*\| 
	\leqslant \frac{1}{7}\|z_0-z^*\|, \label{3:lem3:5}\\
	&\sum_{i=2}^{k}(\sigma_{i}^t + \sigma_{i-1}^t)\|z_{i-1}^t-(z_{i-1}^t)^*\|
	= \sum_{i=2}^{k} 5 \sigma_{i-1}^t\frac{1}{8^{i-1}}\|z_0-z^*\| 
	\leqslant \frac{5\sigma_{1}^t}{4}\|z_0-z^*\|. \label{3:lem3:6}
	\end{align}
	By plugging \eqref{3:lem3:5} and \eqref{3:lem3:6} into \eqref{3:lem2:1}, we have
	\begin{align*}
	\|F(z_k^t)\| \leqslant&~ \bigg(\sqrt{\frac{12}{11}} + \frac{72}{11}\bigg) M_k^t \frac{1}{64^k}\|z_0-z^*\|^2 + (\frac{2}{7} +1) \sqrt{\frac{12}{11}} M_k^t\frac{1}{8^k} \|z_0-z^*\|^2 \nonumber \\
	&~+ \bigg(\sqrt{\frac{12}{11}} +1\bigg) 4^{k-1}\sigma_{1}^t\frac{1}{8^k}\|z_0-z^*\|  
	+ \sqrt{\frac{12}{11}}\frac{1}{8^k}\|DF(z_0)\|\|z_0-z^*\| \\
	&~+ \sigma_1^t\|z_0 -z^*\| +   \frac{5\sigma_{1}^t}{4}\|z_0-z^*\|\\
	\leqslant&~ 8M_k^t \frac{1}{64^k}\|z_0-z^*\|^2
	+ 2M_k^t\frac{1}{8^k}\|z_0-z^*\|^2 +\frac{41}{16}\sigma_1^t\|z_0 -z^*\| \\
	&~+ \sqrt{\frac{12}{11}}\frac{1}{8^k}\|DF(z_0)\|\|z_0-z^*\| ,
	\end{align*}
	where the second inequality is by $\sqrt{\frac{12}{11}} + \frac{72}{11} \leqslant 8, \left(\frac{2}{7}+1\right)\sqrt{\frac{12}{11}} \leqslant 2$ and $\left(\sqrt{\frac{12}{11}} +1\right) 4^{k-1}\frac{1}{8^k} \leqslant \left(\sqrt{\frac{12}{11}} +1\right)\frac{1}{8} \leqslant \frac{5}{16}$.
	By the choice of $K_t$ and $\sigma_{k}^t$, we have $\|F(z_{K_t}^t)\|\leqslant \epsilon$.

	Combining \eqref{3:lem2:9} with \eqref{3:lem3:4} and \eqref{3:lem3:5}, it follows that for any $k\geqslant 1$
		\begin{align*}
			\|z_k^t - z^*\| \leqslant \frac{1}{8}\|z_0-z^*\| + \frac{2}{7}\|z_0-z^*\|+ 2\|z_0-z^*\|=\frac{135}{56}\|z_0-z^*\|.
	\end{align*}
    The proof is then completed by Step 3(e) in Algorithm \ref{alg2}.
\end{proof}
We are now ready to establish the iteration complexity for the FF-CR algorithm. In particular, let $\epsilon > 0$ be any given target accuracy. We denote the first iteration index to achieve $\|F(z_t)\| \leqslant \epsilon$ by
\begin{align*}
T(\epsilon):= \min\{t |~\|F(z_t)\| \leqslant \epsilon \}.
\end{align*}
We provide an upper bound on $T(\epsilon)$ in the following theorem.
\begin{thm}\label{3:thm1}
	Suppose that  Assumptions \ref{ass:func} and \ref{ass:smoothg} hold. Let $\{z_t\}$ be a sequence generated by Algorithm \ref{alg2}. Set 
	\begin{align*}
	&\sigma_k^t = \frac{\epsilon}{41D_t}4^k, 
	N_k^t =  \left\lceil \left( \frac{33\sqrt{3}8^{3-k}H_k^tD_t}{\sigma_{k}^t}\right)^{2/3}  \right\rceil,  \\
	&K_t= \left\lceil \max \left\{\log_{64}\frac{32M_k^tD_t^2}{\epsilon}, \log_8\frac{8M_k^tD_t^2}{\epsilon},\log_8 \frac{4\sqrt{\frac{12}{11}}\|DF(z_0)\|D_t}{\epsilon} \right\}\right\rceil.
	\end{align*} 
	Then, we have  $D_t \leqslant 4 \sqrt{\frac{12}{11}}\|z_0-z^*\|$ for any $t\leqslant T(\epsilon)$ and 
	\begin{align*}
	T(\epsilon)\leqslant  \left \lceil\log_4 \frac{\|z_0-z^*\|}{D_0}\right\rceil,
	\end{align*}
	where $z^*=(x^*,y^*)$ is an optimal solution of \eqref{P} and $D_0 \leqslant \sqrt{\frac{12}{11}}\|z_0-z^*\|$.
\end{thm}
\begin{proof}
	By Step 2 in  Algorithm \ref{alg2}, we initialize $D_0$ by a backtracking routine. According to a similar proof as in \eqref{3:lem2:5}, we have $D_0=\|z_{0,i}-z_0\|\leqslant \sqrt{\frac{12}{11}}\|z_0-z^*\|$.  Notice that  by Lemma \ref{3:lem3}, if $D_t\geqslant \|z_0-z^*\|$, we have $\|F(z_{t+1})\| \leqslant \epsilon$.
	The proof is then completed by using  the definition of $T(\epsilon)$ and Step 4 in Algorithm \ref{alg2} that $D_t = 4^{t}D_0$.
\end{proof}

\begin{rem}\label{rem2}
	Note that in the $t$-th iteration, similar to the proof in Lemma \ref{2:lem3}, for any $k\geqslant 1$, we have $H_0^t=M_0^t=M_0$, $H_{k-1}^t \leqslant H_k^t\leqslant2\rho$, and $M_{k-1}^t\leqslant M_k^t\leqslant2\rho$. Therefore, on the one hand, at iteration $k$, the number of line search steps for $M_k^t$ and $H_k^t$ is upper bounded by $1 + \log_2\frac{M_k^t}{M_{k-1}^t}$ and $N_k^t + \log_2\frac{H_k^t}{H_{k-1}^t}$, respectively. On the other hand,
	Theorem \ref{3:thm1} implies that 
	\begin{align}\label{rem2:1}
	&~\sum_{t=1}^{T(\epsilon)}\sum_{k=1}^{K_t} \left( N_k^t  + \log_2\frac{H_k^t}{H_{k-1}^t} + \log_2\frac{M_k^t}{M_{k-1}^t} +1\right) \nonumber\\
	\leqslant&~ 2\sum_{t=1}^{T(\epsilon)}K_t + \sum_{t=1}^{T(\epsilon)}\sum_{k=1}^{K_t} \left( \frac{33\sqrt{3}8^{3-k}H_k^tD_t}{\frac{\epsilon}{41D_t}4^k}\right)^{2/3} +  2\sum_{t=1}^{T(\epsilon)}\log_2 \frac{2\rho}{M_{0}} \nonumber\\
	\leqslant &~ 2\sum_{t=1}^{T(\epsilon)}K_t +  \sum_{t=1}^{T(\epsilon)} D_t^{4/3} \sum_{k=1}^{K_t} 2^{6-3k} \left( \frac{2706\sqrt{3}\rho}{\epsilon}\right)^{2/3} + 2\sum_{t=1}^{T(\epsilon)}\log_2 \frac{2\rho}{M_{0}} \nonumber\\
	\leqslant&~  T(\epsilon) \left\lceil \max \left\{\log_8\frac{280\rho \|z_0-z^*\|^2}{\epsilon},\log_8 \frac{18\|DF(z_0)\|\|z_0-z^*\|}{\epsilon} \right\}\right\rceil \nonumber\\
	&~+ \frac{512}{7} \left( \frac{2706\sqrt{3}\rho \|z_0-z^*\|^2}{\epsilon}\right)^{2/3} + 2T(\epsilon)\log_2 \frac{2\rho}{M_{0}},	
	\end{align}
	where the last inequality is by $\sum_{t=1}^{T(\epsilon)} D_t^{4/3} = \sum_{t=1}^{T(\epsilon)} D_0^{4/3}4^{4t/3} \leqslant 8 \|z_0-z^*\|^{4/3}$ and $\sum_{k=1}^{K_t} 2^{6-3k} \leqslant \frac{64}{7}. $
	Then, the iteration complexity of the FF-CR algorithm to obtain an $\epsilon$-optimal solution of \eqref{P} with respect to the gradient norm
	is upper bounded by $\mathcal{O}\big( \rho^{2/3}\|z_0-z^*\|^{4/3}\epsilon^{-2/3} \big)$.
\end{rem}
	\begin{rem}
		Compared with the restricted primal-dual gap termination criterion, the gradient norm is a more general termination criterion in optimization problems and is easy to compute in practice \citep{Yoon2021Accelerated,Chen2024Near}. This is particularly important for parameter-free optimization algorithms. 

Moreover, if the FF-CR algorithm obtains an $\epsilon$-optimal solution with respect to the gradient norm, then we can obtain an $\mathcal{O}(\epsilon D)$-optimal solution with respect to the restricted primal-dual gap after no more than $\mathcal{O}\left(\rho^{2/3}\|z_0-z^*\|^{4/3}\epsilon^{-2/3}\right)$ additional iterations of the FF-CR algorithm. We provide the following simple proof. Assume that $\hat{z}=(\hat{x},\hat{y})$ is an $\epsilon$-optimal solution with respect to the gradient norm. According to step 4 of Algorithm \ref{alg2}, i.e., $D_t = 4^{t}D_0$, after $\left\lceil\log_4 \frac{\|z_0-z^*\|}{D_0}\right\rceil $ additional outer iterations of the FF-CR algorithm, we obtain $D_t \geqslant \|z_0-z^*\|$. Then, according to Lemma \ref{3:lem3}, $z_{t+1}$ further satisfies $\|F(z_{t+1})\|\leqslant \epsilon$ and $\|z_{t+1}-z^*\|\leqslant D := \tfrac{135}{56}\|z_0-z^*\|$. Using these bounds, the restricted primal-dual gap can be estimated as
\begin{equation*}
	\textsc{gap}(x_{t+1},y_{t+1};D)\leqslant \max_{x \in \mathbb{B}_D(x^\star),y \in \mathbb{B}_D(y^\star)} \langle z_{t+1} - z, F(\hat{z})\rangle \leqslant 3D\|F(z_{t+1})\|\leqslant 3D \epsilon.
\end{equation*}
Thus, by further combining the bound \eqref{rem2:1}, $z_{t+1}$ becomes an $\mathcal{O}(\epsilon D)$-optimal solution with respect to the restricted primal-dual gap after no more than
$\mathcal{O}\left(\rho^{2/3}\|z_0-z^*\|^{4/3}\epsilon^{-2/3}\right)$
additional iterations of the FF-CR algorithm. This means that the gradient norm termination criterion is stronger than the primal-dual gap in this sense. 
\end{rem}

Note that although the homotopy-continuation cubic regularized Newton algorithm in \citep{Huang2022CubicRN} has the best iteration complexity results in terms of the gradient norm, additional assumptions of error bound conditions are needed. For problems without assuming such error bound conditions, we improve the complexity in \citep{Ostroukhov2020TensorMF} for the convex-concave setting by a logarithmic factor. Moreover, our method achieves better iteration complexity compared to the adaptive second-order optimistic method  in \citep{Jiang2024AdaptiveAO}, which has an iteration complexity of $\mathcal{O}\left((\ell\rho\|z_0-z^*\| + \rho^2 \|z_0-z^*\|^2 )\epsilon^{-1}\right)$ in terms of the gradient norm.

\section{Numerical experiment}
In this section, we compare the numerical performance of the proposed LF-CR algorithm and the FF-CR algorithm with the extragradient (EG) method \citep{Korpelevich1976TheEM}, the Newton-MinMax method \citep{Lin2022ExplicitSM}, the homotopy inexact proximal-Newton extragradient (HIPNEX) algorithm \citep{Marques2024AS}, the optimistic second-order method (OSOM) with line search  \citep{Jiang2022GeneralizedOM} and the parameter-free variant of the adaptive second-order optimistic method (ASOM) \citep{Jiang2024AdaptiveAO} in solving a synthetic minimax optimization problem and an AUC maximization problem. All methods are implemented using MATLAB R2017b on a laptop with Intel Core i5 2.8GHz and 4GB memory.

\subsection{ A synthetic minimax optimization problem}
We consider the following convex-concave minimax optimization problem \citep{Lin2022ExplicitSM}:
\begin{equation}\label{prob:cubic}
\min_{x\in\mathbb{R}^{n}} \max_{y\in\mathbb{R}^{n}} \ f(x, y) = \tfrac{\rho}{6}\|x\|^3 + \langle y, Ax - b\rangle, 
\end{equation}
where $\rho> 0$, the entries of $b \in \mathbb{R}^{n}$ are generated independently from $[-1, 1]$
and $A \in \mathbb{R}^{n \times n}$ is given by
\begin{align*}
	A = \begin{bmatrix}
	1 & -1 & & & \\
	& 1 & -1 & & \\
	& & \ddots & \ddots & \\
	& & & 1 & -1 \\
	& & & & 1
\end{bmatrix}
\end{align*} 
It can be verified that the function $f(x,y)$ is $\rho$-Hessian Lipschitz continuous, and admits a global saddle point $z^\star = (x^\star, y^\star)$ with $x^\star = A^{-1}b$ and $y^\star = -\frac{\rho}{2}\|x^\star\|(A^\top)^{-1}x^\star$. 

In our experiment, the problem parameters are chosen as $n \in \{50, 100\}, \rho = \frac{1}{15n}$.
Let $H_0$ be the initial estimate of Hessian  Lipschitz constant $\rho$ and $D_0$ be the initial estimate of $\|z_0-z^*\|$. For the LF-CR algorithm and the FF-CR algorithm, after selecting the initial point $z_0$, another random point $\tilde{z}_0$ close to the initial point is generated.
	Then, $H_0$ and $D_0$ are generated using the formulas in the two algorithms, and each cubic subproblem is calculated until a high accuracy is achieved.
	For the HIPNEX algorithm, we set $\hat{\sigma}=0$, $\theta=0.3$, and other hyperparameters are determined by formulas $(24)$ and $(25)$ in \cite{Marques2024AS}. For the OSOM algorithm, we set the initial step size $\sigma_0=1$ and the line search hyperparameters $\alpha=0.5$ and $\beta=0.3$. For the ASOM algorithm, we set $\lambda_0=H_0$. We run each algorithm until $\|F(z)\|\leqslant 1e-4.$

\begin{figure}[t]
	\centering
	\subfigure[$n$=50]{\includegraphics[width=0.48\textwidth,trim=20pt 2pt 5pt 50pt]{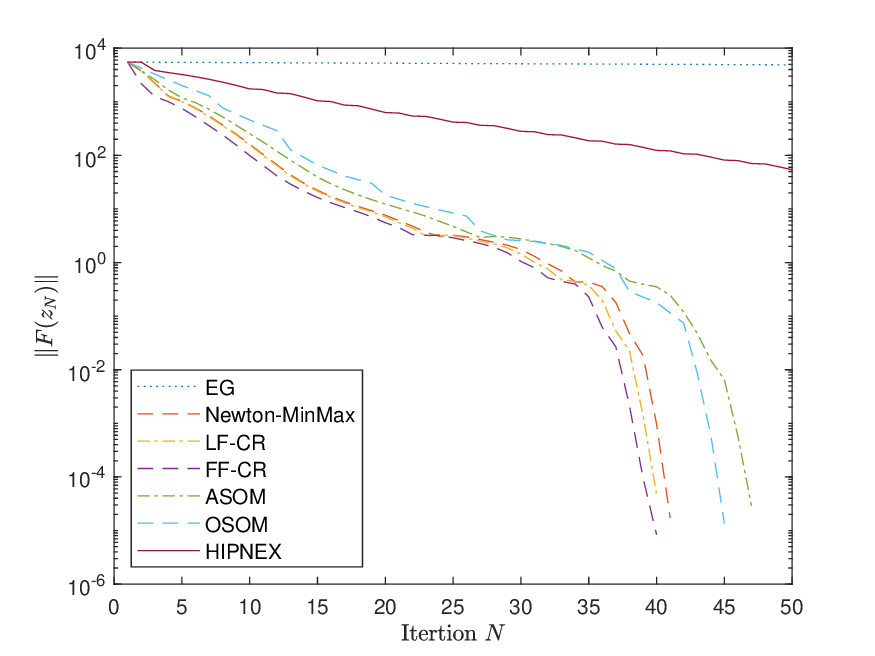}}
	\subfigure[$n$=100]{\includegraphics[width=0.48\textwidth,trim=20pt 2pt 5pt 50pt]{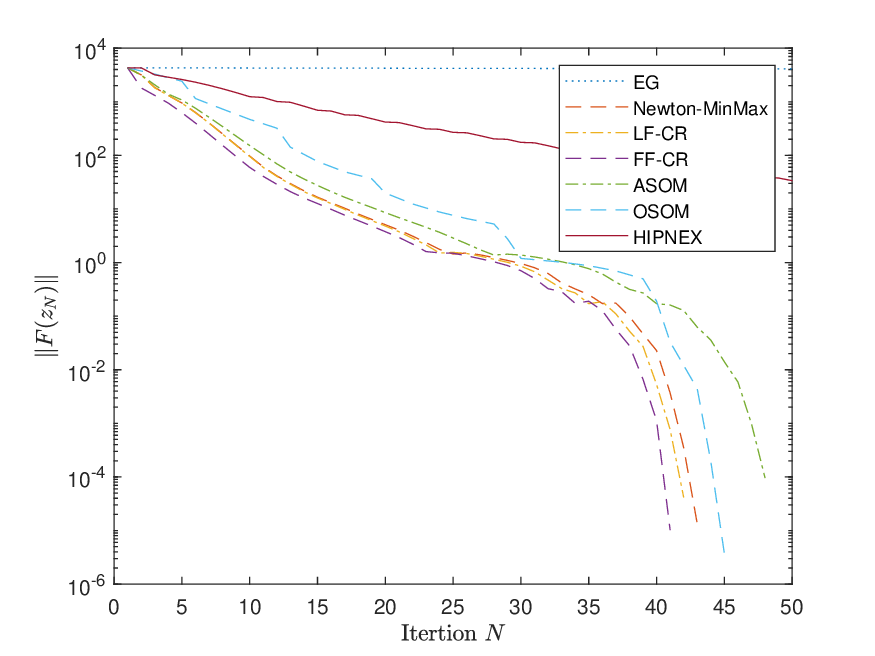}}
	\caption{Numerical results of the tested algorithms for solving synthetic minimax problem}
	\label{fig1}
\end{figure}

Figure \ref{fig1} shows the gradient norm descent curves of the seven algorithms. Obviously, the second-order methods perform better than the first-order method, i.e., the EG method. Among the second-order methods, the LF-CR algorithm and the FF-CR algorithm perform slightly better than the other methods.

\subsection{AUC maximization problem}
 We consider the following minimax formulation for  the area under the receiver operating characteristic curve (AUC) maximization which aims to  find a classifier $\theta \in \mathbb{R}^{n}$  that maximizes AUC score on a set of samples $\{(a_{i}, b_{i})\}_{i = 1}^{M}$, where $a_{i} \in \mathbb{R}^{n}$  and  $b_{i} \in \{1, -1\}$ \citep{Lin2022ExplicitSM,Ying2016StochasticOA,Shen2018TowardsM}:
\begin{align*}
	\min_{x=(\theta, u, v)} \max_{y} & \frac{1 - p}{M} \left\{\sum_{i = 1}^{M} (\langle\theta, a_{i}\rangle - u)^{2} \mathbb{I}_{[b_{i} = 1]}\right\} + \frac{p}{M} \left\{\sum_{i = 1}^{M} (\langle\theta, a_{i}\rangle - v)^{2} \mathbb{I}_{[b_{i} = -1]}\right\} \\
	& + \frac{2(1 + y)}{M} \left\{\sum_{i = 1}^{M} \langle\theta, a_{i}\rangle \bigg(p\mathbb{I}_{[b_{i} = -1]} - (1 - p)\mathbb{I}_{[b_{i} = 1]}\bigg)\right\} + \frac{\rho}{6} \|x\|^{3} - p(1 - p)y^{2},
\end{align*}
where $\rho > 0$ is a scalar, $u, v \in \mathbb{R}$ are auxiliary variables, $\mathbb{I}_{[\cdot]}$ is the indicator function, and $p$  is the proportion of samples with positive label. 

In our experiment, we choose $\rho = \frac{1}{M}$, which is the same as in \citep{Lin2022ExplicitSM}, and use two imbalanced binary classification datasets ``a9a'' $(n=126, M=32561)$, ``w8a'' $(n=300, M=49749)$ from LIBSVM datasets. For the LF-CR algorithm and the FF-CR algorithm,  we generate $H_0$, $M_0$ and $D_0$ using the formulas in both algorithms. For the HIPNEX algorithm, we set $\hat{\sigma}=0$ , $\theta=0.2$, and other hyper-parameters are determined by the formulas $(24)$ and $(25)$ in \cite{Marques2024AS}. For the OSOM algorithm, we set the initial stepsize $\sigma_0=1$, the line-search
hyperparameter $\alpha=0.5$ and $\beta=0.5.$ For the ASOM algorithm, we set $\lambda_0=H_0$.  We run each algorithm until $\|F(z)\|\leqslant 1e-10.$

Figure \ref{fig2} shows the numerical performance of the seven tested algorithms when solving the AUC maximization problem, where the horizontal axis represents the number of iterations  and the vertical axis represents the gradient norm of  $f(x,y)$. The results show that the proposed FF-CR algorithm slightly outperforms the other algorithms.
\begin{figure}[t]
	\centering
	\subfigure[a9a]{\includegraphics[width=0.48\textwidth, trim=20pt 2pt 5pt 50pt]{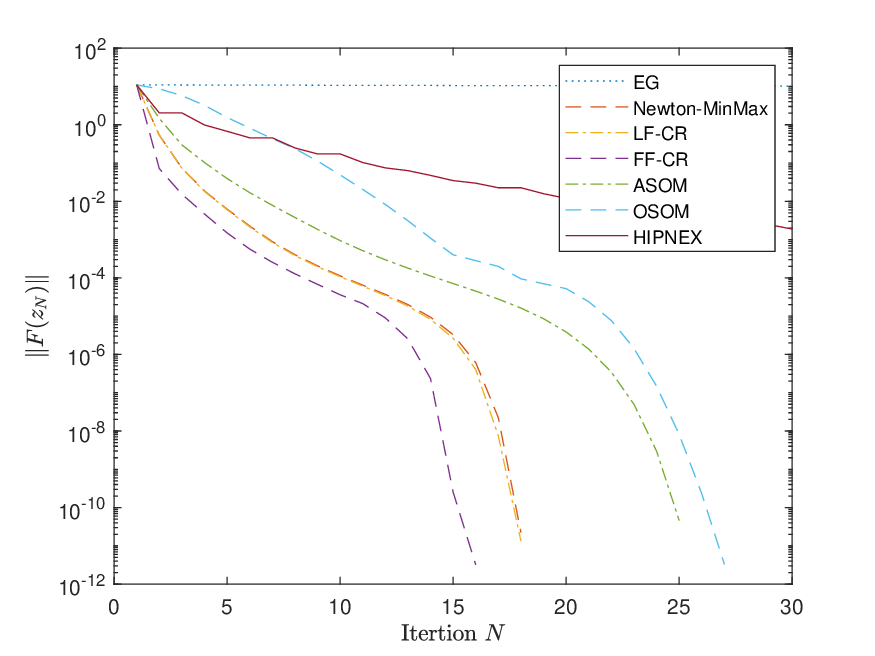}}
	\subfigure[w8a]{\includegraphics[width=0.48\textwidth, trim=20pt 2pt 5pt 50pt]{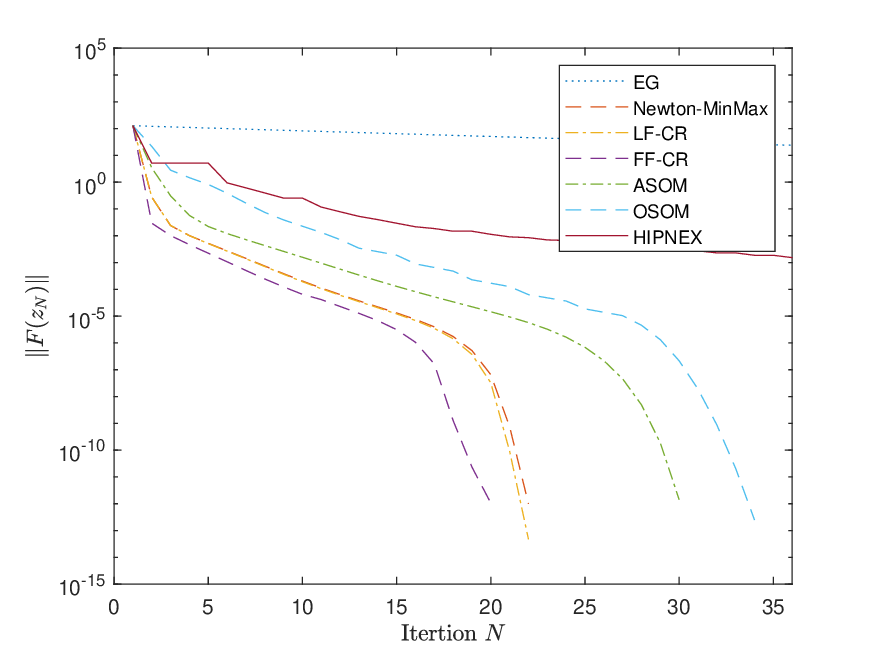}}
	\caption{ Numerical results of the tested algorithms for solving
		AUC maximization problem}
	\label{fig2}
\end{figure}

\section{Conclusions}
In this paper, we propose parameter-free second-order algorithms for unconstrained convex-concave minimax optimization problems. Under a mild assumption that the Hessian is locally Lipschitz continuous, the proposed LF-CR algorithm achieves the iteration complexity of $\mathcal{O}(\rho^{2/3}\|z_0-z^*\|^2\epsilon^{-2/3})$ (resp. $\mathcal{O}(  \rho \|z_0-z^*\|^2 \epsilon^{-1})$) in terms of the $\epsilon$-optimal solution of the restricted primal-dual gap (resp. gradient norm) without prior knowledge of the Lipschitz constants. And  
the proposed FF-CR algorithm achieves  the iteration complexity of  $\mathcal{O}( \rho^{2/3}\|z_0-z^*\|^{4/3}\epsilon^{-2/3})$ in terms of the $\epsilon$-optimal solution of the gradient norm without any prior knowledge of problem parameters. To the best of our knowledge, the proposed FF-CR algorithm is a completely parameter-free second-order algorithm, and its iteration complexity is currently the best in terms of $\epsilon$  under the termination criterion of the gradient norm.

The only disadvantage of the proposed FF-CR algorithm is that it requires the exact optimal solutions of the subproblems. If the subproblems are allowed to be solved inexactly, whether it is still possible to obtain the current iteration complexity of a completely parameter-free algorithm is worth further study.

\vskip 0.2in
\bibliography{FF-CR}

\end{document}